\documentclass[10pt]{article}
\usepackage{amssymb,amsmath}
\usepackage{verbatim}
\usepackage[dvipsnames*,svgnames]{xcolor}
\newtheorem{thm}{Theorem}
\newtheorem{lem}[thm]{Lemma}
\newtheorem{prop}[thm]{Proposition}
\newtheorem{dfn}[thm]{Definition}
\newtheorem{rem}[thm]{Remark}
\newtheorem{cor}[thm]{Corollary}
\newenvironment{proof}[1][Proof]{\textbf{#1.} }{\ \rule{0.5em}{0.5em}}

\let\eps=\varepsilon
\let\a=\alpha
\let\b=\beta
\let\D=\Delta
\let\d=\delta
\let\s=\sigma
\let\g=\gamma
\let\t=\theta
\let\l=\lambda
\let\kap=\kappa
\let\L=\Lambda

\let\ls=\leqslant
\let\gs=\geqslant

\def\arol#1{\stackrel{\mathbf #1}{\longrightarrow}}

\def\dvit{\colon\ }

\def\E{\hbox{\bf E}}

\title{On estimation of the extended Orey index for Gaussian processes}
\date{}
\author{K. Kubilius\\ \\ {\small Vilnius University, Institute of Mathematics and
Informatics,}\\ {\small Akademijos 4, LT-08663, Vilnius, Lithuania}}

\begin{document}

\maketitle

\let\oldthefootnote\thefootnote
\renewcommand{\thefootnote}{\fnsymbol{footnote}}
\footnotetext{E-mail:
\texttt{kestutis.kubilius@mii.vu.lt}}
\let\thefootnote\oldthefootnote

\abstract{Orey suggested the definition of an index for Gaussian process with stationary increments which determines various properties of the sample paths of this process.  We provide an extension of the definition of the Orey index towards a second order stochastic process  which may not have stationary increments and  estimate the Orey index towards a Gaussian process from discrete observations of its sample paths.
\bigskip

\emph{Keywords}: Gaussian process,  Hurst index, fractional Brownian motion, incremental variance function}

\emph{AMS Subject Classification}: primary 60G15; secondary 60G22

\section{Introduction}

Fractional Brownian motion (fBm) is a popular model in financial mathematics, economics and natural sciences. It is well known that the fBm  $B^H$ is the only continuous Gaussian process which is self-similar with stationary increments and that depends only on index $0<H<1$. Moreover, a fBm with Hurst index $H$ is H\"older up to order $H$.

For a real zero-mean Gaussian process with stationary increments, Orey suggested the following definition of the index.

\begin{dfn}[see \cite{or}, \cite{rn0}]\label{oreydef} Let $X$ be a real-valued zero-mean  Gaussian stochastic process with stationary increments and continuous in mean-square sense. Let $\s_X$ be the incremental variance of $X$ given by $\s_X^2(h)=\E[X(t+h)-X(t)]^2$ for $t,h\gs 0$. Define
\begin{equation}\label{orey1}
\widehat\b_*:=\inf\Big\{\b>0\dvit \lim_{h\downarrow 0}\frac{h^\b}{\s_X(h)}=0\Big\}=\limsup_{h\downarrow 0}\frac{\ln\s_X(h)}{\ln h}
\end{equation}
and
\begin{equation}\label{orey2}
\widehat\b^*:=\sup\Big\{\b>0\dvit \lim_{h\downarrow 0}\frac{h^\b}{\s_X(h)}=+\infty\Big\}=\liminf_{h\downarrow 0}\frac{\ln\s_X(h)}{\ln h}\,.
\end{equation}
If $\widehat\b_*=\widehat\b^*$ then $X$ has the Orey index $\widehat\b_*=\widehat\b^*=\b_X$.
\end{dfn}
If a Gaussian process with stationary increments has Orey index then almost all sample paths satisfy a H\" older condition of order $\g$ for each $\g\in (0,\b_X)$ (see Section 9.4 of Cramer and Leadbetter \cite{crle}). For fBm  $B^H$ with the Hurst index $0<H<1$ the Orey index $\b_{B^H}=H$. So we have a class of Gaussian processes with stationary increments depending on the Orey index  $\b_X$.

Recently there have been two extensions of fBm which preserve many properties of fBm, but have no stationary increments except for particular parameter values. One of them is a so called sub-fractional Brownian motion (sfBm) (see \cite{BGT}) and another one is a bifractional Brownian motion (bifBm) (see \cite{HV}, \cite{rt}). Thus it is very natural to extend the definition of the Orey index for Gaussian processes so that it would be possible to consider processes which may not have stationary increments and have the Orey index.

We will provide such extension of the Orey index. As will be shown later, processes  sfBm and bifBm satisfy this extended definition of the Orey index and  are H\"older up to the Orey index. Moreover, for fBm, sfBm, and bifBm, the Orey index coincides with their self-similarity parameter. Therefore it is enough to construct and consider the asymptotic behavior of an estimate of the Orey index instead of estimating parameters of each of the processes under consideration.

Many authors have already considered the asymptotic behavior of the first- and second-order quadratic variations of Gaussian processes (see \cite{begyn1}, \cite{begyn2}, \cite{glad}, \cite{IL}, \cite{kg}, \cite{ma}, \cite{rn1}). The
conditions in those papers are expressed in terms of covariance of a Gaussian process and depend on some parameter $\g\in(0,2)$. If a Gaussian process has the Orey index then conditions on a covariance function may be expressed in its terms. As it will be shown below, the Orey index can be obtained for some well-known Gaussian processes.
Moreover, in order to consider  stochastic differential equations (SDE) driven by processes with a bounded $p$-variation, we should know when the  Riemann-Stieltjes (RS) integral is defined. For Gaussian processes the Orey index helps to obtain these conditions.

The purpose of this paper is to give an extension of the definition of the Orey index for the second order stochastic processes  which may not have stationary increments and to estimate the Orey index for a Gaussian process from discrete observations of its sample paths.

Norvai\v sa \cite{rn2} extends the definition of the Orey index for the second order stochastic processes which may not have stationary increments. He showed that  sfBm and bifBm satisfy this extended definition of the Orey index. In this paper we  give a different extension of the definition of the Orey index as it is more convenient for our purposes. Moreover, it is given in an explicit form.

The paper is organized in the following way. Section 2 contains the definition of the Orey index for a second order stochastic process. The conditions when the second order stochastic process has the Orey index are also given. Moreover, for some well-known Gaussian processes which do not have stationary increments the Orey index is obtained.
Section 3 contains the results on an almost sure asymptotic behavior of the second-order quadratic variations of a Gaussian process. We also verify the obtained conditions for some  well-known Gaussian processes.

\section{Generalized Orey index for the second order stochastic processes}

Let $X = \{X(t)\dvit t \in [0, T ]\}$ be a second order stochastic process with the
incremental variance function $\s_X^2$ defined on $[0, T ]^2 := [0, T ]\times[0, T ]$ with values
\[
\s_X^2(s, t) := \E[X(t)-X(s)]^2,\quad (s, t)\in[0, T ]^2.
\]

Denote by $\Psi$ a class of continuous functions $\varphi\dvit (0,T]\to [0,\infty)$ such that $\lim_{h\downarrow  0}\varphi(h)=0$ and $\lim_{h\downarrow  0}[h\cdot L^3(h)]=0$, where $L(h)=\varphi(h)/h\to\infty$, $h\downarrow 0$. For example, we can take $\varphi(h)=h\cdot\vert\ln h\vert^\a$ or $\varphi(h)=h^{1-\b}$, where $\a>0$, $0<\b<1/3$.
Set
\begin{align}
\g_*:=&\inf\bigg\{\g>0\dvit \lim_{h\downarrow 0}\sup_{\varphi(h)\ls s\ls T-h}\frac{h^\g}{\s_X(s,s+h)}=0\bigg\},\label{oreyeq1}\\
\widetilde\g_*:=&\inf\Big\{\g>0\dvit \lim_{h\downarrow 0}\frac{h^\g}{\s_X(0,h)}=0\Big\}\label{oreyeq2}
\end{align}
and
\begin{align}
\g^*:=&\sup\bigg\{\g>0\dvit \lim_{h\downarrow 0}\inf_{\varphi(h)\ls s\ls T-h}\frac{h^\g}{\s_X(s,s+h)}=+\infty\bigg\},\label{oreyeq3}\\
\widetilde\g^*:=&\sup\Big\{\g>0\dvit \lim_{h\downarrow 0}\frac{h^\g}{\s_X(0,h)}=+\infty\Big\}\,,\label{oreyeq4}
\end{align}
where $\varphi\in\Psi$. Note that $0\ls\widetilde\g^*\ls\widetilde\g_*\ls +\infty$ and $0\ls\g^*\ls\g_*\ls +\infty$.

We give the following extension of the Orey index.
\begin{dfn}\label{oreydef1}
Let $X = \{X(t)\dvit t \in [0, T ]\}$ be a second order stochastic process with the
incremental variance function $\s_X^2$ such that $\sup_{0\ls s\ls T-h}\s_X(s,s+h)\to 0$ as $h\to 0$. If $\g_*=\widetilde\g_*=\g^*=\widetilde\g^*$ for any function $\varphi\in\Psi$, then we say that the process $X$ has the Orey index $\g_X=\g_*=\widetilde\g_*=\g^*=\widetilde\g^*$.
\end{dfn}

\begin{rem}\label{oreydef2} If we consider a real-valued mean zero Gaussian stochastic process with stationary increments and continuous in mean square then the  Orey indices in Definition \ref{oreydef} and Definition \ref{oreydef1} coincide.
\end{rem}

Let us introduce the notions
\begin{align}\label{oreyeq}
\widehat\g_*:=&\limsup_{h\downarrow 0}\sup_{\varphi(h)\ls s\ls T-h}\frac{\ln\s_X(s,s+h)}{\ln h}\quad\mbox{and}\quad \overline\g_*:=\limsup_{h\downarrow 0}\frac{\ln\s_X(0,h)}{\ln h}\,,\\
\widehat\g^*:=&\liminf_{h\downarrow 0}\inf_{\varphi(h)\ls s\ls T-h}\frac{\ln\s_X(s,s+h)}{\ln h}\quad\mbox{and}\quad \overline\g^{\,*}:=\liminf_{h\downarrow 0}\frac{\ln\s_X(0,h)}{\ln h}\,.
\end{align}

It follows from Remark \ref{oreydef2} and (\ref{orey1}), (\ref{orey2}) that $\widetilde\g_*=\overline\g_*$ and   $\widetilde\g^*=\overline\g^{\,*}$. Now we compare the values of $\widehat\g^*$ and $\widehat\g_*$ with $\g^*$ and $\g_*$, respectively,  for a second order stochastic process $X$.

\begin{lem}\label{orey} Let $X = \{X(t)\dvit t \in [0, T ]\}$ be a second order stochastic process with the
incremental variance function $\s_X^2$ such that
\begin{equation}\label{orey0}
\sup_{\varphi(h)\ls s\ls T-h}\s_X(s,s+h)\longrightarrow 0\qquad\mbox{as}\ h\downarrow 0.
\end{equation}
{If $0<\widetilde\g^*\ls\widetilde\g_*< +\infty$, then} $\widehat\g^*=\g^*$ and $\widehat\g_*=\g_*$.
\end{lem}

\begin{proof} The proof of the lemma follows the outlines of calculation of limits of the logarithmic ratios (see Annex A.4 in \cite{tr}). For completeness we give this proof in Appendix.
\end{proof}

For $(s,t)\in[0,T]^2$, $s\neq t$, set
\[
b(s,t):=\frac{\s_X(s,t)}{\kap\vert t-s\vert^\g}-1.
\]

Assume that for some $\g\in(0,1)$ the second order stochastic process $X$ satisfies the conditions:

(C1)\quad $\s_X(0,\d)\asymp\d^{\g}$, i.e., $\s_X(0,\d)$ and $\d^{\g}$ are of the same order as $\d\downarrow 0$;

(C2)\quad there exists a constant $\kappa>0$ such that
\[
\L(\d):=\sup_{\varphi(\d)\ls t\ls T-\d} \sup_{0<h\ls \d}\big\vert b(t,t+h)\big\vert\longrightarrow 0\qquad\mbox{as}\ \d\downarrow 0
\]
for every function $\varphi\in\Psi$.

For $(s,t)\in[0,T]^2$, $s\neq t$, set
\begin{equation}\label{variacija0}
c(s,t):=\frac{\s^2_X(s,t)}{\kap^2\vert t-s\vert^{2\g}}-1.
\end{equation}
It follows from $(C1)$ and $(C2)$ that for any $\varphi\in\Psi$
\begin{align}\label{variacija}
\sup_{0\ls s\ls T-h} \s^2_X(s,s+h)\ls& \sup_{0\ls s\ls \varphi(h)} \s^2_X(s,s+h)+\sup_{\varphi(h)\ls s\ls T-h}
\s^2(s,s+h)\nonumber\\
\ls& 4\sup_{0\ls \d\ls \varphi(h)+h} \s^2_X(0,\d)+\kap^2h^{2\g} \Big(\sup_{\varphi(h)\ls s\ls T-h}\vert c(s,s+h)\vert+1\Big)\nonumber\\
\ls& O\big((\varphi(h))^{2\g}\big) +\kap^2h^{2\g}\big[\L^2(h)+2\L(h)+1\big]\longrightarrow 0\quad\mbox{as}\ h\downarrow 0.
\end{align}
Thus the process $X$ is continuous in quadratic mean for all $s\in[0,T-h]$.
\begin{thm} Assume that for some constant $\g\in(0,1)$ the second order stochastic process $X$ satisfies  conditions $(C1)$ and $(C2)$. Then it has the Orey index that equals  $\g$.
\end{thm}

\begin{proof} Due to Lemma \ref{orey} it is sufficient to show that $\widehat\g_*=\widehat\g^*=\g$ and $\overline\g_*=\overline\g^{\,*}=\g$.

We first note that condition $(C1)$ implies $\overline\g_*=\overline\g^{\,*}=\g$. Since $\s_X(0,\d)\asymp \d^{\g}$, then there exist constants $0<m<M <\infty$ and  $\d_0$ such that
\[
m< \frac{\s_X(0,\d)}{\d^{\g}} <M\qquad \mbox{for all}\ \d < \d_0.
\]
Thus
\[
\g+\frac{\ln M}{\ln \d}=\frac{\ln (M\d^\g)}{\ln \d}<\frac{\ln\s_X(0,\d)}{\ln \d}< \frac{\ln (m\d^\g)}{\ln \d}=\g+\frac{\ln m}{\ln \d}\qquad\mbox{as}\ \d<1,
\]
and we obtain equality $\overline\g_*=\overline\g^{\,*}=\g$.

It remains to prove that $\widehat\g^*=\widehat\g_*$. It follows from conditions $(C1)$ and $(C2)$ that there exists $\d_0$ such that for $\d\ls \d_0< 1$ inequalities $\s_X(s,s+\d)\ls 1/2$ for all $0\ls s\ls T-\d_0$ and $\L(\d)<1/2$ hold. In what follows we suppose that
inequalities hold for $\d\ls \d_0< 1$.

We fix some function $\varphi\in\Psi$. Assume that $-1/2<b(s_0,s_0+\d_0)\ls 0$ for some fixed $s_0\in[\varphi(\d_0),T-\d_0]$. Furthermore, it is known that $-2x\ls\ln(1-x)\ls -x$ for $0\ls x\ls 1/2$. From this inequality we get
\begin{align*}
\ln\s_X(s_0,s_0+\d_0)=& \ln (\kap \d_0^\g)+\ln(1+b(s_0,s_0+\d_0))= \ln (\kap \d_0^\g)+\ln(1-(-b(s_0,s_0+\d_0)))\\
\ls& \ln (\kap \d_0^\g)+b(s_0,s_0+\d_0)\ls \ln(\kap \d_0^\g)+\L(\d_0)
\end{align*}
and
\begin{align*}
\ln\s_X(s_0,s_0+\d_0)\gs& \ln(\kap \d_0^\g)+2 b(s_0,s_0+\d_0)=\ln(\kap \d_0^\g)-2\vert b(s_0,s_0+\d_0)\vert\\
\gs& \ln (\kap \d_0^\g)-2\L(\d_0).
\end{align*}

It is known that $\vert\ln(1+x)\vert\ls x$  for $x\gs 0$. Assume that $0\ls b(s_0,s_0+\d_0)< 1/2$ for some fixed $s_0\in[\varphi(\d_0),T-\d_0]$, then
\begin{align*}
\ln\s_X(s_0,s_0+\d_0)=&\ln(\kap \d_0^\g)+\ln(1+b(s_0,s_0+\d_0))\ls\ln(\kap \d_0^\g)+b(s_0,s_0+\d_0)\\
\ls& \ln(\kap \d_0^\g)+\L(\d_0)
\end{align*}
and
\begin{align*}
\ln\s_X(s_0,s_0+\d_0)=&\ln(\kap \d_0^\g)+\ln(1+b(s_0,s_0+\d_0))\gs \ln(\kap \d_0^\g)-2\vert b(s_0,s_0+\d_0)\vert\\
\gs& \ln(\kap \d_0^\g)-2\L(\d_0).
\end{align*}
Thus for every $s\in[\varphi(\d_0),T-\d_0]$ we obtain
\[
\ln(\kap \d_0^\g)-2\L(\d_0)\ls\ln\s_X(s,s+\d_0)\ls\ln(\kap \d_0^\g)+\L(\d_0).
\]
Consequently,
\begin{align*}
\g+\frac{\ln\kap}{\ln \d_0}-\frac{\L(\d_0)}{\vert\ln \d_0\vert}\ls&\inf_{\varphi(\d_0)\ls s\ls T-\d_0} \frac{\ln\s_X(s,s+\d_0)}{\ln \d_0}\ls \sup_{\varphi(\d_0)\ls s\ls T-\d_0}  \frac{\ln\s_X(s,s+\d_0)}{\ln \d_0}\\
\ls& \g+\frac{\ln\kap}{\ln \d_0}+2\,\frac{\L(\d_0)}{\vert\ln \d_0\vert}\,.
\end{align*}
and both sides of the above inequality approach $\g$ as $\d_0\to 0$. Thus
$\widehat\g_*=\widehat\g^*=\g$.
\end{proof}

\subsection{Subfractional Brownian motion}

We shall prove that sfBm satisfies conditions $(C1)$ and $(C2)$.

\begin{dfn}{\rm (\cite{BGT})}
A \textbf{sub-fractional Brownian motion} with index $H$, $H\in(0,1)$, is a zero-mean Gaussian stochastic process $S^H=(S^H_t, t\gs 0)$ with covariance function
\[
G_H(s,t):=s^{2H} +t^{2H}-\frac{1}{2}\big[(s+t)^{2H}+\vert s-t\vert^{2H}\big].
\]
\end{dfn}

The incremental variance function of sfBm is of the following form
\begin{equation}\label{increm1}
\s_{S^H}^2(s,t)=\E\vert S^H_t-S^H_s\vert^2= \vert t-s\vert^{2H}+(s+t)^{2H}-2^{2H-1}(t^{2H}+s^{2H}).
\end{equation}

Since for any $0\ls s\ls t\ls T$ inequalities (see \cite{BGT})
\begin{align}
&(t-s)^{2H}\ls\s_{S^H}^2(s,t)\ls (2-2^{2H-1})(t-s)^{2H}, \qquad\mbox{if}\quad 0<H<1/2,\label{subf1}\\
&(2-2^{2H-1})(t-s)^{2H}\ls\s_{S^H}^2(s,t)\ls(t-s)^{2H}, \qquad\mbox{if}\quad 1/2<H<1\label{subf2}
\end{align}
hold, then condition $(C1)$ is satisfied.

We get from  (\ref{increm1}) that
\[
\s_{S^H}^2(s,s+h)=h^{2H}+f_s(h),
\]
where
\[
f_s(h):=(2s+h)^{2H}-2^{2H-1}\big[s^{2H} +(s+h)^{2H}\big].
\]
Note that
\[
f_s(0)=f^\prime_s(0)=0.
\]
By Taylor formula we obtain
\begin{align*}
f_s(h)=&f_s(0)+f^\prime_s(0)h+\int_0^h f^{\prime\prime}_s(x)(h-x)\,dx=\int_0^h f^{\prime\prime}_s(x)(h-x)\,dx \\ =&2H(2H-1)\int_0^h\big[(2s+x)^{2H-2}-2^{2H-1}(s+x)^{2H-2}\big](h-x)\,dx.
\end{align*}
From inequality
\begin{align*}
&\big[2^{2H-1}(s+x)^{2H-2}-(2s+x)^{2H-2}\big]=\frac{1}{(s+x)^{2-2H}}\bigg[2^{2H-1}-\bigg(\frac{s+x}{2s+x}\bigg)^{2-2H}\bigg]\\
&\quad=\frac{1}{(s+x)^{2-2H}}\bigg[2^{2H-1}-\bigg(1-\frac{s}{2s+x}\bigg)^{2-2H}\bigg]
\ls \frac{1}{(s+x)^{2-2H}}\big[2^{2H-1}-2^{-1}\big],
\end{align*}
it follows that for $s>0$
\[
\vert f_s(h)\vert\ls (2^{2H}-1)\int_0^h\frac{h-x}{(s+x)^{2-2H}}\,dx\ls \frac{1}{2}\,(2^{2H}-1)s^{2H-2}h^2
\]
and
\begin{align*}
\sup_{\varphi(\d)\ls s\ls T-\d}\sup_{0<h\ls \d}\bigg\vert \frac{\s_{S^H}^2(s,s+h)}{h^{2H}}-1\bigg\vert=&\sup_{\varphi(\d)\ls s\ls T-\d}\sup_{0<h\ls \d} \frac{\vert f_s(h)\vert}{h^{2H}}\\
\ls& \sup_{\varphi(\d)\ls s\ls T-\d}\frac{2^{2H-1}\d^{2-2H}}{s^{2-2H}}\ls \frac{2^{2H-1}}{(L(\d))^{2-2H}}
\end{align*}
for every $\varphi\in\Psi$, where $L(h)=\varphi(h)/h$. So we get condition $(C2)$ with $\kappa=1$.

\begin{rem} The function $\varphi(\d)$  could not be replaced by $\d$ or $0$ in condition (C2). Indeed, e.~g., $H>1/2$. Then
\begin{align*}
&\sup_{0\ls s\ls T-\d}\sup_{0\ls h\ls \d}\vert h^{-2H}f_s(h)\vert\gs\sup_{\d\ls s\ls T-\d}\sup_{0\ls h\ls \d}\vert h^{-2H}f_s(h)\vert\\
&\quad= 2H(2H-1) \sup_{\d\ls s\ls T-\d}\sup_{0\ls h\ls \d} \int_0^h \Big[\frac{2^{2H-1}}{h^{2H}(s+x)^{2-2H}}-\frac{1}{h^{2H}(2s+x)^{2-2H}}\Big](h-x)\,dx\\
&\quad\gs
2H(2H-1) \sup_{\d\ls s\ls T-\d}\sup_{0\ls h\ls \d} \int_0^h \frac{2^{2H-1}-1}{h^{2H}(2s+x)^{2-2H}}\,(h-x)\,dx\\
&\quad\gs H(2H-1) \sup_{\d\ls s\ls T-\d}\sup_{0\ls h\ls \d}\frac{(2^{2H-1}-1)h^{2-2H}}{(2s+h)^{2-2H}}\\
&\quad= H(2H-1)(2^{2H-1}-1)\sup_{\d\ls s\ls T-\d}\frac{\d^{2-2H}}{(2s+\d)^{2-2H}}\\
&\quad= H(2H-1)(2^{2H-1}-1)3^{2H-2}.
\end{align*}
\end{rem}

\subsection{Bifractional Brownian motion}

\begin{dfn}{\rm (\cite{HV})}
A \textbf{bifractional Brownian motion}  $B^{HK}=(B^{HK}_t, t\gs 0)$ with parameters $H\in(0,1)$ and $K\in(0,1]$ is a centered Gaussian process with covariance function
\[
R_{HK}(t,s)=2^{-K}\big((t^{2H} +s^{2H})^K-\vert t-s\vert^{2HK}\big),\qquad s,t\gs 0.
\]
\end{dfn}

The incremental variance function of bifBm is of the following form
\[ 
\s_{B^{H,K}}^2(s,t)=\E\vert B^{H,K}_t-B^{H,K}_s\vert^2= 2^{1-K}\big[\vert t-s\vert^{2HK}-(t^{2H}+s^{2H})^K\big]+t^{2HK}+s^{2HK}.
\] 
Let $H\in(0,1)$ and $K\in(0,1]$. Then
\begin{equation}\label{nelyg17}
2^{-K}\vert t-s\vert^{2HK}\ls\s_{B^{H,K}}^2(s,t)\ls 2^{1-K}\vert t-s\vert^{2HK}
\end{equation}
for all $s,t\in[0,\infty)$ (see \cite{HV}). Thus condition $(C1)$ holds.

Since
\[
\s_{B^{H,K}}^2(s,s+h)=2^{1-K}( h^{2HK}-f_s(h))
\]
with
\[
f_s(h):=\big[s^{2H} +(s+h)^{2H}\big]^K-2^{K-1}\big[s^{2HK}+(s+h)^{2HK}\big],
\]
then $f_s(0)=f^\prime_s(0)=0$ and by Taylor formula we obtain
\[
\frac{\s_{B^{HK}}^2(s,s+h)}{2^{1-K}h^{2HK}}-1=-h^{-2HK}\int_0^h f^{\prime\prime}_s(x)(h-x)\,dx,
\]
where
\begin{align*}
f^{\prime\prime}_s(x)=& 4K(K-1)H^2\big[s^{2H}+(s+x)^{2H}\big]^{K-2}(s+x)^{2(2H-1)}\\ &+2HK(2H-1)\big[s^{2H}+(s+x)^{2H}\big]^{K-1}(s+x)^{2H-2}\\
&-2^K HK(2HK-1)(s+x)^{2HK-2}.
\end{align*}
Note that for $H\gs 1/2$
\[
\frac{(s+x)^{2(2H-1)}}{[s^{2H}+(s+x)^{2H}]^{2-K}}=\bigg[\frac{(s+x)^{2H}}{s^{2H}+(s+x)^{2H}}\bigg]^{2-K}\,(s+x)^{2HK-2} \ls (s+x)^{2HK-2}.
\]
Thus for $s>0$
\begin{align*}
\sup_{0\ls x\ls h}\vert f^{\prime\prime}_s(x)\vert\ls& \frac{4}{s^{2-2HK}}\,{\bf 1}_{\{H\gs 1/2\}}+\frac{4}{(2s^{2H})^{2-K}s^{2(1-2H)}}\,{\bf 1}_{\{H< 1/2\}}\\
&+\frac{2}{(2s^{2H})^{1-K}s^{2-2H}}+\frac{2}{s^{2-2HK}}
\ls \frac{8}{s^{2-2HK}}
\end{align*}
and
\[
\sup_{\varphi(\d)\ls s\ls T-\d}\sup_{0<h\ls \d}\bigg\vert \frac{\s_{B_{HK}}^2(s,s+h)}{2^{1-K}h^{2H}}-1\bigg\vert\ls \sup_{\varphi(\d)\ls s\ls T-\d}\frac{8 \d^{2-2HK}}{s^{2-2HK}}
\ls\frac{8}{(L(\d))^{2-2HK}}
\]
for every $\varphi\in\Psi$. So condition $(C2)$ holds.

\subsection{Ornstein-Uhlenbeck process}\label{OU}

The fractional Ornstein-Uhlenbeck (fO-U) process of the first kind is the unique solution of the following stochastic differential equation
\begin{equation}\label{O-U1}
X_t=x_0-\mu\int_0^t X_s\,ds+\t B^H_t,\qquad t\ls T,
\end{equation}
with $\mu,\t>0$, where $B^H$, $0<H<1$, is a fBm. This equation has an explicit solution
\[
X_t=x_0 e^{-\mu t}+\t\int_0^t e^{-\mu(t-u)}dB^H_u,
\]
where the integral exists as a Riemann-Stieltjes integral for all $t > 0$ (see, e.g., \cite{chm}).

First check condition $(C1)$ for this process. From \cite{chm} we know that
\[
\int_0^t e^{\mu u}dB^H_u=e^{\mu t}B^H_t-\mu \int_0^t e^{\mu u}B^H_u du.
\]
Thus
\begin{align*}
X_t^2=&\bigg(x_0 e^{-\mu t}+\t\int_0^t e^{-\mu(t-u)}dB^H_u\bigg)^2= \bigg(x_0 e^{-\mu t}+\t B^H_t-\t\mu e^{-\mu t}\int_0^t e^{\mu u}B^H_u du\bigg)^2 \\
\ls& 3x_0^2+3\t^2(B^H_t)^2+ 3\t^2e^{-2\mu t}\bigg(\int_0^t  B^H_u d e^{\mu u}\bigg)^2\\
\ls&3x_0^2+3\t^2(B^H_t)^2+3\t^2 e^{-\mu t} \int_0^t  (B^H_u)^2 d e^{\mu u}
\end{align*}
and
\[
\sup_{t\ls T}\E X_t^2\ls 3x_0^2+6\t^2 T^2.
\]
The incremental variance function of $X$ has the following form
\begin{align*}
\s_X^2(t,t+h)=&\mu^2\E\bigg(\int_t^{t+h}X_s\,ds\bigg)^2-2\mu\t\E\bigg([B^H(t+h)-B^H(t)]\int_t^{t+h}X_s\,ds\bigg)\\ &+\t^2\s_{B^H}^2(t,t+h).
\end{align*}
Cauchy-Schwarz inequality yields
\[
\E\bigg(\int_t^{t+h}X_s\,ds\bigg)^2\ls h^2 \sup_{t\ls s\ls t+h}\E X^2_s
\]
and
\begin{align*}
&\E\bigg([B^H(t+h)-B^H(t)]\int_t^{t+h}X_s\,ds\bigg)\ls  \E^{1/2}[B^H(t+h)-B^H(t)]^2 \bigg(h\int_t^{t+h}\E X^2_s\,ds\bigg)^{1/2}\\
&\quad\ls h^{H+1}\Big(\sup_{t\ls s\ls t+h}\E X^2_s\Big)^{1/2}.
\end{align*}
Note that
\[
\big\vert\s_X^2(0,h)-\t^2 h^{2H}\big\vert=\big\vert\s_X^2(0,h)-\t^2\E(B^H_h)^2\big\vert \ls \mu^2 h^2\sup_{t\ls h}\E X_t^2+2\mu\t h^{1+H}\sqrt{\sup_{t\ls h}\E X_t^2}\,.
\]
Condition (C1) follows from these calculations since $\s_X^2(0,h)\asymp h^{2H}$ is equivalent to the requirement that $\s_X^2(0,h)/ h^{2H}$ converges to some finite non-zero limit as $h\downarrow 0$.

Next, for every $\varphi\in\Psi$
\[
\sup_{\varphi(\d)\ls t\ls T-\d} \sup_{0<h\ls \d}\bigg\vert\frac{\s^2_X(t,t+h)}{\t^2 h^{2H}}- 1\bigg\vert\ls \t^{-2}\d^{1-H}\Big[\d^{1-H}\mu^2 \sup_{t\ls T}\E X^2_t+2\mu\t\Big(\sup_{t\ls T}\E X^2_t\Big)^{1/2}\Big] \longrightarrow 0
\]
as $\d\downarrow 0$. It follows from the inequality
\[
\bigg\vert\frac{\s_X(t,t+h)}{\t h^{H}}- 1\bigg\vert = \bigg\vert\frac{\s^2_X(t,t+h)}{\t^2 h^{2H}}- 1\bigg\vert\Big/ \bigg\vert\frac{\s_X(t,t+h)}{\t h^{H}}+ 1\bigg\vert\ls \bigg\vert\frac{\s^2_X(t,t+h)}{\t^2 h^{2H}}- 1\bigg\vert
\]
that condition $(C2)$ is satisfied.

\subsection{Fractional Brownian bridge}

The fractional Brownian bridge is defined in $[0,T]$ as
\begin{equation}\label{bridge}
X_t^H=B^H_t-\frac{t^{2H}+T^{2H}-\vert t-T\vert^{2H}}{2T^{2H}}\, B^H_T,
\end{equation}
where  $B^H$, $0<H<1$, is a fBm on the interval $[0,T]$ (see \cite{gsv}).

Now we verify condition $(C1)$. The incremental variance function of $X^H$ has the following form
\[
\s_{X^H}^2(t,t+h)=h^{2H}-\frac{1}{4T^{2H}}\, f^2_t(h),
\]
where
\[
f^2_t(h):=\big[(t+h)^{2H}-t^{2H}-\vert t+h-T\vert^{2H}+\vert t-T\vert^{2H}\big]^2.
\]
It is easy to verify that
\[
\frac{f^2_0(h)}{h^{2H}}=\frac{[h^{2H}-\vert T-h\vert^{2H}+T^{2H}]^2}{h^{2H}}\longrightarrow 0\qquad \mbox{as}\ h\downarrow 0.
\]
Thus
\[
\s_{X^H}^2(t,t+h)\asymp h^{2H}.
\]
So condition $(C1)$ is satisfied.

Assume that $H<1/2$. Since
\[
\big\vert(t+h)^{2H}-t^{2H}\big\vert\ls h^{2H}\quad\mbox{and}\quad \big\vert(T-t-h)^{2H}-(T-t)^{2H}\big\vert\ls h^{2H},
\]
then for every $\varphi\in\Psi$
\[
\sup_{\varphi(\d)\ls t\ls T-\d} \sup_{0<h\ls \d}\bigg\vert\frac{\s^2_{X^H}(t,t+h)}{h^{2H}}- 1\bigg\vert =\frac{1}{4T^{2H}}\sup_{\varphi(\d)\ls t\ls T-\d} \sup_{0<h\ls \d}\frac{f^2_t(h)}{h^{2H}} \ls T^{-2H}\d^{2H}.
\]

Assume that $H\gs 1/2$. Then $f_t(0)=0$ and we obtain from Taylor formula that
\[
\frac{\s_{X^H}^2(t,t+h)}{h^{2H}}-1=-\frac{1}{4T^{2H}h^{2H}}\bigg(\int_0^h f^{\prime}_t(x)\,dx\bigg)^2,
\]
where
\[
f^{\prime}_t(x)= 2H\big[(t+x)^{2H-1}-(T-t-x)^{2H-1}\big].
\]
Thus for every $\varphi\in\Psi$ and $H\gs 1/2$ we get
\[
\sup_{\varphi(\d)\ls t\ls T-\d} \sup_{0<h\ls \d}\bigg\vert\frac{\s^2_{X^H}(t,t+h)}{h^{2H}}- 1\bigg\vert\ls \frac{\d^{2-2H}}{4T^{2H}}\cdot 4H^2T^{4H-2}=H^2T^{2H-2}\d^{2-2H}.
\]

\section{The convergence of the second order quadratic variation of process $X$ along arbitrary partition}

Let $\pi_n=\{0=t^n_0<t^n_1<\cdots<t^n_{N_n}=T\}$, $T>0$, be a sequence of partitions of the interval $[0,T]$, where $(N_n)$ is an increasing sequence of natural numbers. Define
\[
m_n=\max_{1\ls k\ls N_n}\D^n_k t,\qquad p_n=\min_{1\ls k\ls N_n}\D^n_k t,\qquad \D^n_k t=t^n_k-t^n_{k-1}.
\]

\begin{dfn} A sequence of partitions $(\pi_n)_{n\in\mathbb{N}}$ is regular if  $m_n=p_n=T N_n^{-1}$ for all $n\in\mathbb{N}$ or, equivalently,
$t^n_k=\frac{kT}{N_n}$ for all $n\in\mathbb{N}$ and all $k\in\{0,\ldots,N_n\}$.
\end{dfn}

Usually in practice observations of the process are available at discrete regular time intervals. However, it may happen that the  part of observations is lost, resulting in observations at arbitrary time intervals. Therefore we define the second order quadratic variations of Gaussian processes along arbitrary partitions.

\begin{dfn} The second order quadratic variations of Gaussian processes $X$ with Orey index $\g$ along the partitions $(\pi_n)_{n\in\mathbb{N}}$  is
defined by
\[
V_{\pi_n}^{(2)}(X,2)=2\sum_{k=1}^{N_n-1} \frac{\Delta^n_{k+1} t(\D^{(2)n}_{ir,k}X)^2}{
(\Delta^n_k t)^{\g+1/2}(\Delta^n_{k+1} t)^{\g+1/2}[\Delta^n_k t+\Delta^n_{k+1} t]}\,,
\]
where
\[
\D^{(2)n}_{ir,k}X^n_k=\Delta^n_k t X(t^n_{k+1})
+\Delta^n_{k+1} t X(t^n_{k-1})-(\Delta^n_k t+\Delta^n_{k+1} t)X(t^n_{k}).
\]
\end{dfn}

If the sequence $(\pi_n)_{n\in\mathbb{N}}$ is regular then one has
\[
V^{(2)}_{N_n}(X,2)=(T^{-1}N_n)^{2\g-1}\sum_{k=1}^{N_n-1}\big(\D^{(2)}_{n,k}X\big)^2,\qquad \D^{(2)}_{n,k}X=X(t^n_{k+1})-2X(t^n_{k})
+X(t^n_{k-1})\,.
\]

To study the almost sure convergence of the second order quadratic variation of $X$ we need additional assumptions on the sequence $(\pi_n)_{n\in\mathbb{N}}$.

\begin{dfn}{\rm (see \cite{begyn1})}\label{begyn} Let $(\ell_k)_{k\gs 1}$ be a sequence of real numbers in the interval
$(0,\infty)$. We say that $(\pi_n)_{n\in\mathbb{N}}$ is a sequence of partitions
with asymptotic ratios $(\ell_k)_{k\gs 1}$ if it satisfies the following
assumptions:

1. There exists $c\gs 1$ such that $m_n\ls c p_n$ for all $n$.

2.
$
\lim_{n\to\infty}\max_{1\ls k\ls N_n-1}\bigg\vert \frac{\D^n_k t}{\D^n_{k+1} t}-\ell_k\bigg\vert=0.
$

The set $\mathcal{L}=\{\ell_1;\ell_2;\ldots;\ell_k;\ldots\}$ will be called the
range of the asymptotic ratios of the sequence $(\pi_n)_{n\in\mathbb{N}}$.
\end{dfn}
It is clear that if the sequence $(\pi_n)_{n\in\mathbb{N}}$ is regular, then it is a
sequence with asymptotic ratios  $\ell_k = 1$ for all $k\gs 1$.

\begin{dfn}{\rm (see \cite{begyn1})}
The function $g:(0,\infty)\to\mathbb{R}$ is invariant on $\mathcal{L}$ if for all $\ell,\hat\ell\in\mathcal{L}$, $g(\ell)=g(\hat\ell)$.
\end{dfn}

\begin{dfn}{\rm (see \cite{dieu})} A function $f : [a, b] \to \mathbb{R}$ is called regulated
provided there is a sequence $(f_n)_{n\gs 1}$ of step functions which converges uniformly to $f$.
\end{dfn}

\begin{prop}\label{prop1} Let $X=\{X(t): t\in[0,T]\}$, $T>0$, be a zero mean second order process satisfying conditions $(C1)$ and $(C2)$. Let $(\pi_n)_{n\in\mathbb{N}}$ be a sequence of partitions with asymptotic ratios $(\ell_k)_{k\gs 1}$  and range of the asymptotic ratios $\mathcal{L}$. If the function
\[
g(\lambda)=\frac{1+\lambda^{2\g-1}-(1+\l)^{2\g-1}}{\lambda^{\g-1/2}}
\]
is invariant on $\mathcal{L}$ or if the sequence of step functions $\ell_n(t)$, i.e. $\ell_n(t)=\ell_k$ on $(t^n_k,t^n_{k+1})$, $0\ls k\ls N_n-1$ and $\ell_0=\ell_1$, converges to regulated function $\ell(t)$ on the interval $[0,T]$,  then
\[
\E V_{\pi_n}^{(2)}(X,2)\longrightarrow 2\kap^2\int_0^T g(\ell(t))\,dt \qquad\mbox{as}\ n\to\infty.
\]
\end{prop}

\proof Rewrite the expectation of each increment from the second order variation in the following way
\begin{align*}
\E(\D^{(2)n}_{ir,k}X)^2=&(\D^n_k t)^2\s^2_X(t_k^n, t_{k+1}^n)+(\D^n_{k+1} t)^2\s^2_X(t_{k-1}^n, t_k^n)\\
&+\D^n_kt\cdot \D^n_{k+1}t\big[\s^2_X(t_k^n, t_{k+1}^n)-\s^2_X(t_{k-1}^n,t_{k+1}^n)+\s^2_X(t_{k-1}^n, t_k^n)\big]\\
=&[\D^n_kt+\D^n_{k+1}t]\big[\D^n_k t\cdot\s^2_X(t_k^n,t_k^n+\D^n_{k+1} t)+\D^n_{k+1} t\cdot\s^2_X(t_{k-1}^n,t_{k-1}^n+\D^n_k t)\big]\\
&-\D^n_k t\cdot \D^n_{k+1} t\cdot\s^2_X(t_{k-1}^n,t_{k-1}^n+\D^n_k t+\D^n_{k+1} t)\\
=&I^{(1)}_k-I^{(2)}_k+I^{(3)}_k,
\end{align*}
where
\begin{align*}
I^{(1)}_k:=&[\D^n_k t+\D^n_{k+1} t]\big\{\D^n_k t\big[\s^2_X(t_k^n,t_{k+1}^n)-\kap^2( \D^n_{k+1} t)^{2\g}\big]\\
&+\D^n_{k+1} t\big[\s^2_X(t_{k-1}^n, t_k^n)-\kap^2(\D^n_k t)^{2\g}\big]\big\},\\
I^{(2)}_k:=&\D^n_k t\cdot \D^n_{k+1} t\big[\s^2_X(t_{k-1}^n,t_{k+1}^n)-\kap^2(\D^n_k t+\D^n_{k+1} t)^{2\g}\big],\\
I^{(3)}_k:=&\kap^2[\D^n_k t+\D^n_{k+1} t]\D^n_k t \cdot \D^n_{k+1} t  \big\{(\D^n_{k+1} t)^{2\g-1}+(\D^n_k t)^{2\g-1} -(\D^n_k t+\D^n_{k+1} t)^{2\g-1}\big\}.
\end{align*}
Set
\[
\mu_k^n=[\D^n_k t+\D^n_{k+1} t](\D^n_{k+1} t)^{\g+1/2}(\D^n_k t)^{\g+1/2}\quad\mbox{and}\quad \ell_k^n=\frac{\D^n_k t}{\D^n_{k+1} t}\,.
\]
Then
\begin{align*}
I^{(1)}_k=& \kap^2[\D^n_k t+\D^n_{k+1} t]\D^n_k t\cdot\D^n_{k+1} t\big[(\D^n_{k+1} t)^{2\g-1} c(t_k^n,t_{k+1}^n)+(\D^n_k t)^{2\g-1}c(t_{k-1}^n,t_k^n)\big]\\
=&\kap^2\mu_k^n\big[(\ell_k^n)^{1/2-\g}c(t_k^n,t_{k+1}^n)+(\ell_k^n)^{\g-1/2}c(t_{k-1}^n,t_k^n)\big],\\
I^{(2)}_k=& \kap^2\mu_k^n(\D^n_k t)^{1/2-\g}(\D^n_{k+1} t)^{1/2-\g}(\D^n_k t+\D^n_{k+1} t)^{2\g-1}c(t_{k-1}^n,t_{k+1}^n)\\
=&\kap^2\mu_k^n (\ell_k^n)^{1/2-\g}(1+\ell_k^n)^{2\g-1}c(t_{k-1}^n,t_{k+1}^n)
\end{align*}
and
\begin{align*}
I^{(3)}_k= \kap^2\mu_k^n\big((\ell_k^n)^{1/2-\g}+(\ell_k^n)^{\g-1/2}-(\ell_k^n)^{1/2-\g}(1+\ell_k^n)^{2\g-1}\big)
=\kap^2\mu_k^n\, g(\ell_k^n),
\end{align*}
where the function $c(s,t)$ is defined in (\ref{variacija0}). Further, we note  that
\begin{align}\label{lygybe}
\E V_{\pi_n}^{(2)}(X,2)=&2\sum_{k=1}^{\tau_n+1}\frac{\D^n_{k+1} t\cdot\E(\D^{(2)n}_{ir,k}X)^2}{\mu_k^n} +2\sum_{k=\tau_n+2}^{N_n-1} \frac{\D^n_{k+1} t\cdot\E(\D^{(2)n}_{ir,k}X)^2}{\mu_k^n}\nonumber\\
=&2\sum_{k=1}^{\tau_n+1}\frac{\D^n_{k+1} t\cdot\E(\D^{(2)n}_{ir,k}X)^2}{\mu_k^n} +2\kap^2\sum_{k=\tau_n+2}^{N_n-1}\D^n_{k+1} t\cdot J_k\nonumber\\
&+2\kap^2\sum_{k=\tau_n+2}^{N_n-1}\D^n_{k+1} t\cdot g(\ell_k^n),
\end{align}
where $\tau_n=[\varphi(m_n)N_n]$, $[a]$ is an integer part of a real number $a$,
\begin{align*}
J_k=(\ell_k^n)^{1/2-\g}\big[c(t_k^n,t_{k+1}^n)+(\ell_k^n)^{2\g-1}c(t_{k-1}^n,t_k^n)-(1+\ell_k^n)^{2\g-1} c(t_{k-1}^n,t_{k+1}^n)\big].
\end{align*}

Now we estimate the first term in the right-hand side of (\ref{lygybe}).  Note that
\begin{align}
\tau_n\ls& \frac{\varphi(m_n)}{p_n}\,T\ls c L(m_n)T,\qquad 2p_n^{2\g+2}\ls\mu_k^n\ls 2m_n^{2\g+2}, \label{nelyg10a}\\
\sum_{k=1}^{\tau_n+1}\Delta t_{k+1}^n\ls&  \varphi(m_n)\frac{m_n}{p_n}\,T+m_n\ls
cT\varphi(m_n)+m_n\ls c(T+1)\varphi(m_n)\quad \mbox{if}\ L(m_n)>1.\label{nelyg11}
\end{align}
We get from $(C1)$, $(C2)$ and inequalities (\ref{nelyg10a}), (\ref{nelyg11}) that
\begin{align*}
&2\sum_{k=1}^{\tau_n+1}\frac{\Delta_{k+1}^n t\cdot\E(\D^{(2)n}_{ir,k}X)^2}{\mu_k^n}\\
&\quad \ls\frac{4c^3(T+1)\varphi(m_n)}{p_n^{2\g}}\,\max_{1\ls k\ls \tau_n+2}\s_X^2(t^n_{k-1},t^n_k) \ls\frac{16c^3(T+1)\varphi(m_n)}{p_n^{2\g}}
\sup_{1\ls k\ls \tau_n+2}\s_X^2(0,t^n_k)\\
&\quad
=\frac{16c^3(T+1)\varphi(m_n)}{p_n^{2\g}}\, O\big(L^{2\g}(m_n)m_n^{2\g}\big)
\ls 16c^3(T+1)\varphi(m_n)\,O\big( L^{2\g}(m_n)\big)
\end{align*}
as $m_n\downarrow 0$. We obtain from the properties of function $\varphi$  that the right
hand side of the above inequality tends to zero as $m_n\downarrow 0$.

Next, since $[\varphi(m_n)N_n]+1\gs \varphi(m_n)$, we get that the second term of equality (\ref{lygybe}) can be estimated as
\begin{align*}
&\sum_{k=\tau_n+2}^{N_n-1}\D^n_{k+1} t\cdot J_k\\
&\quad\ls \max_{\tau_n+1\ls k\ls N_n-1}\big\vert c(t_k^n,t_{k+1}^n)\big\vert\sum_{k=\tau_n+2}^{N_n-1}\D^n_{k+1} t \big[ (\ell_k^n)^{1/2-\g}+(\ell_k^n)^{\g-1/2}\big]\\
&\qquad+\max_{\tau_n+2\ls k\ls N_n-1}\big\vert c(t_{k-1}^n,t_{k+1}^n)\big\vert\sum_{k=\tau_n+2}^{N_n-1}\D^n_{k+1} t\cdot (\ell_k^n)^{1/2-\g}(1+\ell_k^n)^{2\g-1}\\
&\quad\ls T\sup_{\varphi(m_n)\ls t\ls T-m_n}\sup_{0<h\ls m_n}\big\vert c(t,t+h)\big\vert \max_{1\ls k\ls N_n}\big[(\ell_k^n)^{1/2-\g}+(\ell_k^n)^{\g-1/2}\big]\\
&\qquad+T \sup_{\varphi(m_n)\ls t\ls T-2m_n}\sup_{0<h\ls m_n}\big\vert c(t,t+2h)\big\vert \max_{1\ls k\ls N_n}\big[(\ell_k^n)^{1/2-\g}(1+\ell_k^n)^{2\g-1}\big]\\
&\quad\ls T\big[\L^2(m_n)+2\L(m_n)\big] \max_{1\ls k\ls N_n}\big[(\ell_k^n)^{1/2-\g}+(\ell_k^n)^{\g-1/2}\big]\\
&\qquad+T\big[\L^2(2m_n)+2\L(2m_n)\big] \max_{1\ls k\ls N_n}\big[(\ell_k^n)^{1/2-\g}(1+\ell_k^n)^{2\g-1}\big]\\
&\quad\ls 2Tc\big[\L^2(m_n)+2\L(m_n)\big]
+T(1+c)c\big[\L^2(2m_n)+2\L(2m_n)\big].
\end{align*}
Thus the second term of equality (\ref{lygybe}) tends to zero as $n\to\infty$.

It still remains to investigate asymptotic behavior of the third term of equality (\ref{lygybe}).  If the function $g$ is invariant on $\mathcal{L}$, then
\begin{align*}
2\kap^2\sum_{k=\tau_n+2}^{N_n-1}\D^n_{k+1} t\cdot g(\ell_k^n)
=&2\kap^2\sum_{k=\tau_n+2}^{N_n-1}\D^n_{k+1} t\cdot[ g(\ell^n_k)-g(\ell_k)]\\
&+ 2\kap^2 T g(\ell_1) - 2\kap^2g(\ell_1)\sum_{k=0}^{\tau_n+1}\D^n_{k+1} t.
\end{align*}
Assumption 1 of Definition \ref{begyn} implies that  $(\ell_k)_{k\gs 1}\subset [c^{-1},c]$. Since the derivative of the function $g$ is bounded on $[c^{-1},c]$ by $2c^{3/2}$, then
\[
\vert g(\ell^n_k)-g(\ell_k)\vert \ls 2c^{3/2}\vert \ell^n_k-\ell_k\vert\,.
\]
Thus
\[
2\kap^2\sum_{k=\tau_n+2}^{N_n-1}\D^n_{k+1} t\cdot g(\ell_k^n) \longrightarrow 2\kap^2 g(\ell_1)T\quad\mbox{as}\ n\to\infty
\]
by assumption 2 of Definition \ref{begyn} and the inequality (\ref{nelyg11}).

Assume that the sequence of step functions $\ell_n(t)$ converges uniformly to $\ell(t)$ on the interval $[0,T]$. Then $\ell_n(t), \ell(t)\in[c^{-1},c]$ and
\[
\vert g(\ell_n(t))-g(\ell(t))\vert\ls 2c^{3/2} \sup_{0\ls t\ls T}\vert \ell_n(t)-\ell(t)\vert\longrightarrow 0\quad\mbox{as}\ n\to\infty,
\]
i.e. the sequence  $g(\ell_n(t))$ converges uniformly to $g(\ell(t))$ on $[0,T]$ and $g(\ell(t))$ is regulated function. Thus
\begin{align*}
2\kap^2\sum_{k=\tau_n+2}^{N_n-1}\D^n_{k+1} t\cdot g(\ell^n_k)
=&2\kap^2\sum_{k=\tau_n+2}^{N_n-1}\D^n_{k+1} t\cdot[ g(\ell^n_k)-g(\ell_k)]+2\kap^2\int_0^T g(\ell_n(t))\,dt\\
&-2\kap^2\sum_{k=0}^{\tau_n+1}\D^n_{k+1} t\cdot g(\ell_k)\longrightarrow 2\kap^2\int_0^T g(\ell(t))\,dt
\end{align*}
since regulated functions are Riemann integrable and
\begin{align*}
&\sum_{k=\tau_n+2}^{N_n-1}\D^n_{k+1} t\cdot\vert g(\ell^n_k)-g(\ell_k)\vert\ls 2c^{3/2}T\max_{1\ls k\ls N_n-1}\vert \ell^n_k-\ell_k\vert\,,\\
&g(\ell_k)\ls c^{1/2}(1+c)\qquad \mbox{for all}\ k\gs 1.
\end{align*}
Consequently, in both cases we obtain that
\[
\E V_{\pi_n}^{(2)}(X,2)\longrightarrow 2\kap^2\int_0^T g(\ell(t))\,dt\qquad\mbox{as}\ n\to\infty.\qquad\endproof
\]

\begin{cor}\label{cor0} Let $(\pi_n)_{n\in\mathbb{N}}$ be a sequence of regular partitions of the interval $[0,T]$, $T>0$, and let $X=\{X(t): t\in[0,T]\}$, $T>0$, be a zero mean second order process satisfying conditions $(C1)$ and $(C2)$. Then
\[
\E V_{N_n}^{(2)}(X,2)\longrightarrow \kap^2(4-2^{2\g}) T\qquad \mbox{as}\ n\to\infty.
\]
\end{cor}
\proof For regular subdivision we have that $\ell_k=1$. Thus $g(\l)=2-2^{2\g-1}$ and the statement of the corollary follows immediately from Proposition \ref{prop1}.\qquad
\endproof

Now we formulate a slightly more general version of Corollary \ref{cor0}.

\begin{prop}\label{prop2} Let $(\pi_n)_{n\in\mathbb{N}}$ be a sequence of regular partitions of the interval $[0,T]$, $T>0$. Assume that condition $(C1)$ is fulfilled for some constant $\g\in(0,1)$ and there exists a continuous bounded function $g_0: (0,T)\to\mathbb{R}$ such that
\begin{equation}\label{asump}
\lim_{h\to 0+}\sup_{\varphi(h)\ls t\ls T-h}\bigg\vert \frac{{\bf E}\big(X_{t+h}-2X_t+X_{t-h}\big)^2}{h^{2\g}}-g_0(t)\bigg\vert=0.
\end{equation}
Then
\[
\E V_{N_n}^{(2)}(X,2)\longrightarrow \int_0^T g_0(t)\, dt\quad \mbox{as}\ n\to\infty.
\]
\end{prop}

\proof Note that
\begin{align}\label{nelyg10}
&\bigg\vert \E V_{N_n}^{(2)}(X,2)- \int_0^T g_0(t)\, dt\bigg\vert\nonumber \\
&\quad\ls \bigg(\frac{T}{N_n}\bigg)^{1-2\g} \sum_{k=1}^{\tau_n} \E(\D^{(2)}_{n,k}X)^2 + \frac{T}{N_n}\sum_{k=\tau_n+1}^{N_n-1}\bigg\vert\frac{\E(\D^{(2)}_{n,k}X)^2}{T^{2\g} N_n^{-2\g}} -g_0\Big(\frac{kT}{N_n}\Big) \bigg\vert\nonumber\\
&\qquad
+\bigg\vert \int_0^T g_0(t)\, dt-\frac{T}{N_n}\sum_{k=\tau_n+1}^{N_n-1} g_0\Big(\frac{kT}{N_n}\Big)\bigg\vert\,,
\end{align}
{where $\tau_n=[\varphi(TN_n^{-1})N_n]$. We get from condition $(C1)$ that
\begin{align*}
\max_{1\ls k\ls \tau_n+1}\s^2(t^n_{k-1},t^n_k)\ls&  2\sup_{1\ls k\ls \tau_n+1} \s^2(0,t^n_k)
=O\big((TN^{-1}_n(\tau_n+1))^{2\g}\big)
=O\big((\varphi(TN^{-1}_n))^{2\g}\big).
\end{align*}}
Thus
\begin{align*}
\bigg(\frac{T}{N_n}\bigg)^{1-2\g} \sum_{k=1}^{\tau_n} \E(\D^{(2)}_{n,k}X)^2
\ls& 4 T\bigg(\frac{T}{N_n}\bigg)^{-2\g}\varphi\bigg(\frac{T}{N_n}\bigg)\max_{1\ls k\ls \tau_n+1}\s^2(t^n_{k-1},t^n_k)\\
=&4 T\bigg(\frac{T}{N_n}\bigg)^{-2\g}\varphi\bigg(\frac{T}{N_n}\bigg) O\bigg(\bigg(\varphi\bigg(\frac{T}{N_n}\bigg)\bigg)^{2\g}\bigg)
\end{align*}
and the first term in inequality (\ref{nelyg10}) tends to zero as $n\to \infty$.

Assumption (\ref{asump}) yields
\begin{align*}
&\max_{\tau_n+1\ls k\ls N_n-1}\bigg\vert\frac{\E(\D^{(2)}_{n,k}X)^2}{T^{2\g}N_n^{-2\g}} -g_0\Big(\frac{kT}{N_n}\Big) \bigg\vert\\
&\quad\ls \sup_{\varphi(TN^{-1}_n)\ls t\ls T-TN^{-1}_n}\bigg\vert \frac{{\bf E}\big(X_{t+TN^{-1}_n}-2X_t+X_{t-TN^{-1}_n}\big)^2}{(TN^{-1}_n)^{2\g}}-g_0(t)\bigg\vert \longrightarrow 0 \qquad \mbox{as}\ n\to \infty.
\end{align*}
The third term of the right hand side of (\ref{nelyg10}) also converges towards $0$ as $n\to \infty$ that is a consequence of classical results for Riemann sums and inequality
\[
\frac{T}{N_n}\sum_{k=1}^{\tau_n}\bigg\vert g_0\Big(\frac{kT}{N_n}\Big)\bigg\vert \ls  \sup_{0\ls t\ls T} \vert g_0(t)\vert\, \varphi\bigg(\frac{T}{N_n}\bigg).
\]
\endproof

\begin{thm}\label{thm2} Assume that conditions of Proposition \ref{prop1} are satisfied and the partition $\pi_n$ is such that $p_n=o(\ln^{-1} n)$.  Moreover assume that $X$ is a Gaussian process with the Orey index $\g$ and
\begin{equation}\label{nelyg8}
\max_{1\ls k\ls N_n-1}\sum_{j=1}^{N_n-1}\vert d^{(2)n}_{jk}\vert\ls C p_n^{2+2\g},
\end{equation}
for some constant $C$ and every sequence of partitions $(\pi_n)$ of
the interval $[0,T]$, where $d^{(2)n}_{jk}=\E (\D_{ir,j}^{(2)n} X \D_{ir,k}^{(2)n} X)$, $1\ls j,k\ls n$. Then
\[
V_{\pi_n}^{(2)}(X,2)\longrightarrow 2\kap^2\int_0^T g (\ell(t))\,dt\quad\mbox{a.s.}\qquad \mbox{as}\ n\to\infty.
\]
\end{thm}
\proof The proof of the theorem follows the outlines of the proof of Theorem 4 in \cite{begyn1}.

\begin{cor}\label{cor} Let $(\pi_n)_{n\in\mathbb{N}}$ be a sequence of regular partitions of the interval $[0,T]$, $T>0$. Assume that $X$ is a Gaussian process satisfying conditions $(C1)$ and $(C2)$ and having the Orey index $\g$. Moreover, assume that
\begin{equation}\label{nelyg9}
\max_{1\ls k\ls N_n-1}\sum_{j=1}^{N_n-1}\vert d^{(2)n}_{jk}\vert\ls C \bigg(\frac{T}{N_n}\bigg)^{2\g}
\end{equation}
for some constant $C$,  and every sequence of partitions $(\pi_n)$ of
the interval $[0,T]$, where $d^{(2)n}_{jk}=\E (\D_{n,j}^{(2)} X \D^{(2)}_{n,k} X)$, $1\ls j,k\ls N_n-1$. Then
\[
V_{N_n}^{(2)}(X,2)\longrightarrow \kap^2(4-2^{2\g}) T\qquad\mbox{a.s.}\quad \mbox{as}\ n\to\infty.
\]
\end{cor}
\proof For regular partition $\pi_n$ condition (\ref{nelyg8}) transforms to (\ref{nelyg9}).

\begin{thm}\label{thm3} Assume that conditions of Proposition \ref{prop2} are satisfied. Moreover, assume that inequality (\ref{nelyg9}) holds, then
\[
V_{N_n}^{(2)}(X,2)\longrightarrow \int_0^T g_0(t)\, dt\qquad\mbox{a.s.}\quad \mbox{as}\ n\to\infty.
\]
\end{thm}
\proof The proof of the theorem evidently follows from Proposition \ref{prop2} and arguments used to prove Theorem \ref{thm2}.

\begin{rem}  Generally speaking, the function $\varphi(h)$  could not be replaced by $h$ in assumption $(\ref{asump})$. Indeed, let us consider sfBm $X$ with $H\neq 1/2$.  Observe that following equality
\begin{align*}
{\bf E}\big(X_{t+h}-2X_t+X_{t-h}\big)^2
=& (4-2^{2H})h^{2H}-2^{2H-1}(t+h)^{2H}-3\cdot 2^{2H}t^{2H} \\
&-2^{2H-1}(t-h)^{2H}+2(2t+h)^{2H}+2(2t-h)^{2H}
\end{align*}
holds. Set
\[
\l_t(h):=\E\big(X_{t+h}-2X_t+X_{t-h}\big)^2- (4-2^{2H})h^{2H}
\]
and note that $\l_t(0)=\l^\prime_t(0)=\l^{\prime\prime}_t(0)=\l^{(3)}_t(0)=0$. The Taylor formula yields
\[
\l_t(h)=\int_0^h \frac{(h-x)^3}{3!}\,\l^{(4)}_t(x)\,dx, \qquad\forall\ h\ls t\ls T-h,
\]
where
\begin{align*}
\lambda_t^{(4)}(x)=& C_H\Big(2\big[(2t+x)^{2H-4}+(2t-x)^{2H-4}\big] -2^{2H-1}\big[(t+x)^{2H-4}+(t-x)^{2H-4}\big]\Big),\\
C_H=&2H(2H-1)(2H-2)(2H-3).
\end{align*}
Note that
\begin{align*}
&\sup_{h\ls t\ls T-h}\bigg\vert \frac{{\bf E}\big(X_{t+h}-2X_t+X_{t-h}\big)^2}{h^{2H}}-(4-2^{2H})\bigg\vert\\
&\quad=\sup_{h\ls t\ls T-h}\bigg\vert \int_0^h \frac{(h-x)^3}{3! h^{2H}}\,\l^{(4)}_t(x)\,dx\bigg\vert
\gs \bigg\vert \int_0^h \frac{(h-x)^3}{3! h^{2H}}\,\l^{(4)}_h(x)\,dx\bigg\vert\,.
\end{align*}
After a change of variable $y=\frac{h-x}{ah+bx}$ with certain constants $a$ and $b$, we obtain equality
\begin{align*}
h^{-2H}\int_0^h (h-x)^3 \l^{(4)}_h(x)\,dx=&2\cdot 3^{2H}C_H \int_0^{1/2} y^3(1+y)^{-2H-1}dy+2C_H \int_0^{1/2} y^3(1-y)^{-2H-1}dy\\
&-2^{4H-1}C_H \int_0^1  y^3(1+y)^{-2H-1}dy-2^{2H-2}H^{-1}C_H.
\end{align*}
All these integrals are finite and don't depend on $h$. Moreover,
\begin{equation}\label{contrexample}
\lim_{h\to 0+}\sup_{h\ls t\ls T-h}\bigg\vert \frac{ \E\big(X_{t+h}-2X_t+X_{t-h}\big)^2}{h^{2H}} -(4-2^{2H})\bigg\vert>0.
\end{equation}

On the other hand, assumption (\ref{asump}) is satisfied for sfBm.
From inequality
\begin{align*}
&\sup_{\varphi(h)\ls t\ls T-h}\bigg\vert \frac{\E\big(X_{t+h}-2X_t+X_{t-h}\big)^2}{h^{2H}}-(4-2^{2H})\bigg\vert \\
&\quad\ls h^{-2H}\sup_{\varphi(h)\ls t\ls T-h}\sup_{0\ls x\ls h}\vert \l^{(4)}_t(x)\vert \int_0^h (h-x)^3\,dx\\
&\quad\ls \vert C_H\vert\cdot h^{4-2H}\sup_{\varphi(h)\ls t\ls T-h}\bigg(\frac{2}{(2t)^{4-2H}}+\frac{2}{(2t-h)^{4-2H}} +\frac{2^{2H-1}}{t^{4-2H}} +\frac{2^{2H-1}}{(t-h)^{4-2H}}\bigg) \\
&\quad\ls \vert C_H\vert\cdot  h^{4-2H} \bigg(\frac{2}{(2\varphi(h))^{4-2H}}+\frac{2}{(2\varphi(h)-h)^{4-2H}} +\frac{2^{2H-1}}{\varphi(h)^{4-2H}} +\frac{2^{2H-1}}{(\varphi(h)-h)^{4-2H}}\bigg) \\
&\quad\ls \vert C_H\vert\cdot \bigg[\bigg(\frac{h}{\varphi(h)}\bigg)^{4-2H} +\frac{2}{(2L(h)-1)^{4-2H}}+2^{2H-1}\bigg(\frac{h}{\varphi(h)}\bigg)^{4-2H} +\frac{2^{2H-1}}{(L(h)-1)^{4-2H}}\bigg]
\end{align*}
we obtain the required assertion.
\end{rem}

\subsection{Bifractional Brownian motion}

We shall prove that the conditions of Theorem \ref{thm2} are satisfied for bifBm. The bifBm satisfies conditions $(C1)$ and $(C2)$. So it suffices to verify the inequality (\ref{nelyg8}).

Following the outlines of the proof of Theorem 4 of Begyn \cite{begyn1}, we divide the study of the asymptotic properties of $d^{(2)n}_{jk}$ into three steps, according to the value of $k-j$.

If $j=k$ then (\ref{nelyg17}) yields
\begin{align}\label{subf4}
d_{kk}^{(2)n}
\ls& 2\big[(\D^n_k t)^2\E(\D^n_{k+1} B^{HK})^2+(\D^n_{k+1} t)^2\E(\D^n_k B^{HK})^2\big]\nonumber\\
\ls& 2^{2-K}\big[(\D^n_k t)^2\vert t_{k+1}-t_k\vert^{2HK}+(\D_{n,k+1} t)^2\vert t_k-t_{k-1}\vert^{2HK}\big]\nonumber\\
\ls& 2^{3-K} m_n^{2+2HK}.
\end{align}

By using the Cauchy-Schwarz inequality we get
\begin{align}\label{subf5}
\big\vert d_{jk}^{(2)n}\big\vert\ls \E^{1/2}\big\vert(\D^{(2)n}_{ir,j} B^{HK})\big\vert^2\cdot \E^{1/2}\big\vert(\D^{(2)n}_{ir,k}B^{HK})\big\vert^2
\ls 2^{3-K} m_n^{2+2HK}
\end{align}
for $1\ls k-j\ls 2$ and
\begin{align}
d^{(2)n}_{j1}\ls& 2^{3-K} m_n^{2+2HK}\quad\mbox{for}\ 1\ls j\ls N_n-1,\\
d^{(2)n}_{1k}\ls& 2^{3-K} m_n^{2+2HK}\quad\mbox{for}\ 1\ls k\ls N_n-1.
\end{align}

Now consider the case $\vert j-k\vert\gs 3$. By symmetry of $d^{(2)n}_{jk}$ one can take $j-k\gs 3$.
Note that for $j\neq 1$ and $k\neq 1$ equality
\begin{align*}
d^{(2)n}_{jk}= \int_{t^n_j}^{t^n_{j+1}}du \int^{t^n_j}_{t^n_{j-1}}dv \int_v^u dw \int_{t^n_k}^{t^n_{k+1}}dx\int^{t^n_k}_{t^n_{k-1}}dy \int_y^x \frac{\partial^4 R_{HK}}{\partial s^2\partial t^2}(w,z)\,dz
\end{align*}
holds. The fourth order mixed partial derivative of the covariance function $R_{HK}(s,t)$ is of the following form
\begin{align*}
\frac{\partial^4 R_{HK}}{\partial s^2\partial t^2}(s,t)
=&-\frac{2HK(2H-1)(2HK-2)(2HK-3)}{2^K\vert s-t\vert^{2(2-KH)}}\nonumber\\
&+\frac{ K(K-1)(K-2)(K-3)(2H)^4}{2^K}\,(st)^{4H-2}\big(s^{2H}+t^{2H}\big)^{K-4}\nonumber\\
&+\frac{ K(K-1)(2H)^2(2H-1)}{2^K}\,\big[(K-2)(2H)+(2H-1)\big](st)^{2H-2}\big(s^{2H}+t^{2H}\big)^{K-2}
\end{align*}
for each $s,t>0$ such that $s\neq t$. Since $2s^Ht^H\ls s^{2H}+t^{2H}$ and $K-2<0$, $K-4<0$ it follows that
\begin{align*}
(st)^{2H-2}\big(s^{2H}+t^{2H}\big)^{K-2}\ls& 2^{K-2}(st)^{KH-2}\\
(st)^{4H-2}\big(s^{2H}+t^{2H}\big)^{K-4}\ls& 2^{K-4}(st)^{KH-2}.
\end{align*}
Thus
\[
\bigg\vert\frac{\partial^4 R^{HK}}{\partial s^2\partial t^2}(s,t)\bigg\vert\ls\frac{C_1}{\vert s-t\vert^{2(2-KH)}} +\frac{C_2}{(st)^{2-KH}}
\]
and
\begin{align}\label{nelyg12}
\vert d^{(2)n}_{jk}\vert
\ls&\int_{t^n_j}^{t^n_{j+1}}du \int^{t^n_j}_{t^n_{j-1}}dv \int_v^u dw \int_{t^n_k}^{t^n_{k+1}}dx\int^{t^n_k}_{t^n_{k-1}}dy \int_y^x \frac{C_1}{\vert w-z\vert^{2(2-KH)}}\, dz\nonumber\\
&+\int_{t^n_j}^{t^n_{j+1}}du \int^{t^n_j}_{t^n_{j-1}}dv \int_v^u dw \int_{t^n_k}^{t^n_{k+1}}dx\int^{t^n_k}_{t^n_{k-1}}dy \int_y^x \frac{C_2}{(wz)^{2-KH}}\, dz\nonumber\\
=:&I^{n,1}_{jk}+I^{n,2}_{jk},
\end{align}
where constants $C_1$ and $C_2$ depends on $H$ and $K$. Inequality
\[
\vert w-z\vert\gs t^n_{j-1}-t^n_{k+1}=\sum_{i=k+2}^{j-1}\D_{n,i} t\gs (j-k-2)p_n
\]
on the integration set imply
\begin{equation}\label{nelyg13}
I^{n,1}_{jk}\ls\frac{4C_1 m_n^6}{(j-k-2)^{2(2-HK)}p_n^{2(2-HK)}}\ls \frac{4C_1 c^6 p_n^{2+2HK}}{(j-k-2)^{2(2-HK)}}\,,
\end{equation}
where $c$ is a constant defined in Definition \ref{begyn}. {Moreover,}
\begin{equation}\label{nelyg14}
\sum_{j-k\gs 3}^{n-1}\frac{1}{(j-k-2)^{2(2-HK)}}\ls \sum_{j=1}^\infty \frac{1}{j^{2(2-KH)}} <\infty.
\end{equation}

Now we estimate $I^{n,2}_{jk}$. {By modifying the computations above we similarly find that}
\begin{align}\label{nelyg15}
I^{n,2}_{jk}\ls& \frac{4C_2 m_n^6}{(t_{j-1}t_{k-1})^{2-KH}}=\frac{4C_2 m_n^6}{(t_{k-1}\sum_{i=k}^{j-1}\D_i t +t_{k-1}^2)^{2-KH}}\nonumber\\
\ls&\frac{4C_2 m_n^6}{p_n^{2-KH}((t_{j-1}-t_{k-1}) +t_{k-1})^{2-KH}}
\ls \frac{4C_2 c^6 p_n^{4+KH}}{(t_{j-1}-t_{k-1})^{2-KH}}\nonumber\\
\ls&  4C_2 c^6 \,\frac{p_n^{2+2KH}}{(j-k)^{2-KH}}\,.
\end{align}
Note that
\begin{equation}\label{nelyg16}
\sum_{j-k\gs 3}^{N_n-1}\frac{1}{(j-k)^{2-KH}}\ls \sum_{j=1}^\infty \frac{1}{j^{2-KH}} <\infty.
\end{equation}
The inequality (\ref{nelyg8}) follows from  inequalities (\ref{nelyg12})-(\ref{nelyg16}).

\subsection{Subfractional Brownian motion}

We recall that conditions $(C1)$ and $(C2)$ are satisfied for sfBm. So the statement of Theorem \ref{thm2} is satisfied if inequality (\ref{nelyg8}) holds. To prove inequality (\ref{nelyg8}), we apply similar arguments as for bifBm.

If $j=k$ or $1\ls k-j\ls 2$ then (\ref{subf1}) and (\ref{subf2}) yields
\[
d_{jk}^{(2)n}\ls 8 m_n^{2+2H}.
\]
The same inequality holds  for $d_{j1}^{(2)n}$, $1\ls j\ls N_n-1$ and $d_{1k}^{(2)n}$, $1\ls k\ls N_n-1$.

The fourth order mixed partial derivative of the covariance function $G_H(s,t)$ is of the following form
\[
\frac{\partial^4 G_H}{\partial s^2\partial t^2}(s,t)=-H(2H-1)(2H-2)(2H-3)\bigg[\frac{1}{\vert s-t\vert^{2(2-H)}}+\frac{1}{(s+t)^{2(2-H)}} \bigg].
\]
for each $s,t>0$ such that $s\neq t$. Note that $(s+t)^{2(2-H)}\gs \vert s-t\vert^{2(2-H)}$ if $s\neq t$.

Thus
\[
\bigg\vert \frac{\partial^4 G_H}{\partial s^2\partial t^2}(s,t)\bigg\vert\ls \frac{2H(2H-1)(2H-2)(2H-3)}{\vert s-t\vert^{2(2-H)}}
\]
and
\[
\vert d^{(2)n}_{jk}\vert\ls\frac{4C_H m_n^6}{(j-k-2)^{2(2-H)}p_n^{2(2-H)}}\ls \frac{4C_H c^6 p_n^{2+2H}}{(j-k-2)^{2(2-H)}}
\]
for $j-k\gs 3$, $2\ls k\ls N_n-1$, where $C_H=2H(2H-1)(2H-2)(2H-3)$, $c$ is a constant defined in Definition \ref{begyn}. So, we have
\begin{align}\label{subf3}
\max_{2\ls k\ls N_n-1}\sum_{j-k\gs 3}d^{(2)n}_{jk}\ls& 4C_H c^6 p_n^{2+2H}\max_{2\ls k\ls N_n-1}\sum_{j-k\gs 3}\frac{1}{(j-k-2)^{2(2-H)}} \nonumber\\
\ls& 4C_H c^6 p_n^{2+2H}\sum_{j=1}^\infty\frac{1}{j^{2(2-H)}}\ls C p_n^{2+2H}
\end{align}
for some constant $C$. This proves (\ref{nelyg8}).

\subsection{Ornstein-Uhlenbeck process}

Before to proof the inequality (\ref{nelyg8}) we  give an auxiliary lemma. In its formulation we shall use the notion $O_r$. Let $(a_n)$ be a sequence of real numbers. The notion of symbol $Y_n=O_r(a_n)$, $a_n\downarrow 0$, means that there exists a.s. finite r.v. $\varsigma$ with the property $\vert Y_n\vert\ls \varsigma\cdot a_n$.
\begin{lem}\label{lem}
Let $X$ be the solution of equation (\ref{O-U1}). Then
\[
\big\vert V^{(2)}_{\pi_n}(X,2)-\t^2 V^{(2)}_{\pi_n}(B^H,2)\big\vert=O_r(m_n^{1-2\eps})
\]
for every $0<\eps<1/2\wedge H$.
\end{lem}
\begin{proof} It is evident that
\[
\D^{(2)n}_{ir,k} X=-\mu\bigg(\D^n_k t\int_{t^n_k}^{t^n_{k+1}} X_s\,ds-\D^n_{k+1} t\int^{t^n_k}_{t^n_{k-1}} X_s\,ds\bigg)+\t \D^{(2)n}_{ir,k} B^H.
\]
For simplicity, we denote $X_k=X(t^n_k)$ and $B^H_k=B^H(t^n_k)$. After simple calculations we get the estimate
\begin{align*}
\sup_{t^n_k\ls s\ls t^n_{k+1}} \vert X_s-X_k\vert\ls& \mu(\D^n_{k+1} t)\sup_{t\ls T}\vert X_t\vert +\t\sup_{t^n_k\ls s\ls t^n_{k+1}} \vert B^H_s-B^H_k\vert\\
\ls& \mu\, m_n \sup_{t\ls T}\vert X_t\vert+\t G^{H,H-\eps}_T m_n^{H-\eps},
\end{align*}
where
\begin{equation}\label{holder}
G^{H,H-\eps}_T:= \sup_{\substack{s\ne t\\s,t\ls T}}{\frac{\vert  B^H_t-B^H_s\vert}{\vert t-s\vert^{H-\eps}}}<\infty\qquad \mbox{for every}\ 0<\eps<H.
\end{equation}
Thus
\begin{align*}
&\bigg(\D^n_k t\int_{t^n_k}^{t^n_{k+1}}( X_s-X_k)\,ds-\D^n_{k+1} t\int^{t^n_k}_{t^n_{k-1}} ( X_s-X_k)\,ds \bigg)^2\\
&\quad\ls 2m_n^3 \int_{t^n_k}^{t^n_{k+1}} ( X_s-X_k)^2\,ds+2m_n^3 \int^{t^n_k}_{t^n_{k-1}} ( X_s-X_k)^2\,ds\\
&\quad\ls 2m_n^4\Big(\sup_{t^n_k\ls s\ls t^n_{k+1}} ( X_s-X_k)^2+\sup_{t^n_{k-1}\ls s\ls t^n_k} ( X_k-X_s)^2\Big)\\
&\quad\ls 8m_n^{4+2H-2\eps}\Big(\mu^2 m_n^{2-2H+2\eps} \sup_{t\ls T} X_t^2+\t^2(G^{H,H-\eps}_T)^2 \Big)
\end{align*}
and
\begin{align*}
&\bigg\vert\bigg(\D^n_k t\int_{t^n_k}^{t^n_{k+1}} X_s\,ds-\D^n_{k+1} t\int^{t^n_k}_{t^n_{k-1}} X_s\,ds\bigg) \D^{(2)n}_{ir,k} B^H\bigg\vert\\
&\quad=\bigg\vert\bigg(\D^n_k t\int_{t^n_k}^{t^n_{k+1}}( X_s-X_k)\,ds-\D_{n,k+1} t\int^{t^n_k}_{t^n_{k-1}} ( X_s-X_k)\,ds \bigg) \D^{(2)n}_{ir,k} B^H\bigg\vert\\
&\quad \ls 2 m_n^{2+H-\eps}\Big(\mu\, m_n^{1-H+\eps}\sup_{t\ls T}\vert X_t\vert+\t G^{H,H-\eps}_T\Big)\cdot 2 m_n G^{H,H-\eps}_T m_n^{H-\eps}\\
&\quad=4 m_n^{3+2H-2\eps}\Big(\mu\, m_n^{1-H+\eps}\sup_{t\ls T}\vert X_t\vert+\t G^{H,H-\eps}_T\Big)\cdot  G^{H,H-\eps}_T.
\end{align*}
We get from the obtained inequalities and definition of $V^{(2)}_{\pi_n}(\cdot,2)$ that
\begin{align*}
\big\vert V^{(2)}_{\pi_n}(X,2)-\t^2 V^{(2)}_{\pi_n}(B^H,2)\big\vert \ls& 8c^{2+2H} m_n^{2-2\eps}\Big(\mu^2 m_n^{2-2H+2\eps} \sup_{t\ls T} X_t^2+2\t^2(G^{H,H-\eps}_T)^2 \Big)T\\
&+ 4c^{2+2H} m_n^{1-2\eps}\Big(\mu m_n^{1-H+\eps}\sup_{t\ls T}\vert X_t\vert+\t G^{H,H-\eps}_T\Big)\cdot  G^{H,H-\eps}_T T\\
=&O_r(m_n^{1-2\eps}).
\end{align*}\end{proof}

As in previous cases it is enough to verify condition (\ref{nelyg8}) of Theorem \ref{thm2} for fBm $B^H$. The following inequality
\[
\bigg\vert\frac{\partial^4 F_H}{\partial s^2\partial t^2}(s,t)\bigg\vert\ls \frac{
H\vert (2H-1)(2H-2)(2H-3)\vert}{\vert s-t\vert^{2(2-H)}}\,,
\]
holds for the covariance function $F_H(s,t)$ of $B^H$. Applying similar arguments as for sfBm we obtain
\[
\max_{1\ls k\ls N_n-1}\sum_{j=1}^{N_n-1}\vert d^{(2)n}_{jk}\vert\ls C p_n^{2+2H}.
\]
From Lemma \ref{lem} and inequality above we get the statement of Theorem \ref{thm2}.

\subsection{Fractional Brownian bridge}

For brevity, we rewrite the fractional Brownian bridge $X_t^H$ given by (\ref{bridge}) as follows:
\[
X_t^H=B^H_t-g(t,T),
\]
where
\[
g(t,T)=\frac{t^{2H}+T^{2H}-\vert t-T\vert^{2H}}{2T^{2H}}\, B^H_T.
\]
It is evident that
\[
\D^{(2)n}_{ir,k} X^H=\D^{(2)n}_{ir,k} B^H-\D^{(2)n}_{ir,k} g(\cdot,T),
\]
where
\begin{align*}
\D^{(2)n}_{ir,k} g(\cdot,T)=&\D_k t\,\frac{(t^{2H}_{k+1}-t_k^{2H})-(\vert T-t_{k+1}\vert^{2H}-\vert T-t_k\vert^{2H})}{2T^{2H}}\, B^H_T\\
&-\D_{k+1} t\,\frac{(t^{2H}_k-t_{k-1}^{2H})-(\vert T-t_k\vert^{2H}-\vert T-t_{k-1}\vert^{2H})}{2T^{2H}}\, B^H_T.
\end{align*}
Since
\[
\vert \D^{(2)n}_{ir,k} g(\cdot,T)\vert\ls \frac{4 m^{1+2H}_n}{2T^{2H}}\, 2^{(2H-1)\lor 0}\, \vert B^H_T\vert\ls \frac{4 m^{1+2H}_n}{T^{2H}}\, \vert B^H_T\vert,
\]
then
\[
V_{\pi_n}^{(2)}(g(\cdot,T),2)\ls 2 T \frac{16 m^{2+4H}_n}{T^{4H} 2p_n^{2+2H}}\, \vert B^H_T\vert=16T^{1-4H} c^{2+4H}p_n^{2H}\, \vert B^H_T\vert\,.
\]
We get from the obtained inequalities and definition of $V^{(2)}_{\pi_n}(\cdot,2)$ that
\begin{align*}
\big\vert V^{(2)}_{\pi_n}(X^H,2)- V^{(2)}_{\pi_n}(B^H,2)\big\vert =&\bigg\vert V_{\pi_n}^{(2)}(g(\cdot,T),2) -4\sum_{k=1}^{N_n-1} \frac{\Delta^n_{k+1} t\D^{(2)n}_{ir,k}B^H \D^{(2)n}_{ir,k} g(\cdot,T)}{
(\Delta^n_k t)^{H+1/2}(\Delta^n_{k+1} t)^{H+1/2}[\Delta^n_k t+\Delta^n_{k+1} t]}\bigg\vert \\
\ls&  V_{\pi_n}^{(2)}(g(\cdot,T),2)+2\,\frac{4 m^{1+2H}_n G^{H,H-\eps}_T m_n^{1+H-\eps}}{T^{2H}p_n^{2+2H}}\, \vert B^H_T\vert\\
\ls&  V_{\pi_n}^{(2)}(g(\cdot,T),2)+8c^{2+2H}\,\frac{G^{H,H-\eps}_T m_n^{H-\eps}}{T^{2H}}\, \vert B^H_T\vert=O_r(m_n^{H-\eps})
\end{align*}
for $0<\eps<H$, where $G^{H,H-\eps}_T$ is defined in (\ref{holder}). By using similar arguments as in previous subsection we get the statement of Theorem \ref{thm2}.

\section{On the estimation of Orey index for arbitrary partition}

Let $(\pi_n)_{n\gs 1}$ be a sequence of partitions of $[0,T]$ such that $0=t^n_0<t^n_1<\cdots<t^n_{m(n)}=T$ for all $n\gs 1$. Assume that we have two sequences of partitions $(\pi_{i(n)})_{n\gs 1}$ and $(\pi_{j(n)})_{n\gs 1}$ of $[0,T]$ such that $\pi_{i(n)}\subset \pi_{j(n)}\subseteq \pi_n$, $i(n)< j(n)\ls m(n)$, for all $n\in\mathbb{N}$, where $\pi_{i(n)}=\{0=t^n_0<t^n_{i(1)}<t^n_{i(2)}<\cdots<t^n_{i(n)}=T\}$ and $\pi_{j(n)}=\{0=t^n_0<t^n_{j(1)}<t^n_{j(2)}<\cdots<t^n_{j(n)}=T\}$. {Set
\[
\D^n_{i(k)} t=t^n_{i(k)}-t^n_{i(k-1)},\qquad m_{i(n)}=\max_{1\ls k\ls i(n)}\D^n_{i(k)} t,\qquad p_{i(n)}=\min_{1\ls k\ls i(n)}\D^n_{i(k)} t.
\]
Moreover, assume that  $p_{j(n)}\neq m_{i(n)}$ and $m_{i(n)}\ls c p_{i(n)}$, for all $i(n)$, $n\gs 1$, $c\gs 1$. Note that $p_{j(n)}\ls p_{i(n)}$. }

Let $X$ be a Gaussian process with Orey index $\g\in(0,1)$. Set
\[
V_{\pi_{i(n)}}^{(2)}(X,2)=2\sum_{k=1}^{i(n)-1} \frac{\D^n_{i(k+1)} t(\D^{(2)n}_{ir,k}X)^2}{
(\D^n_{i(k)} t)^{\g+1/2}(\D^n_{i(k+1)} t)^{\g+1/2}[\D^n_{i(k)} t+\D^n_{i(k+1)} t]}\,,
\]
where
\[
\D^{(2)n}_{ir,i(k)}X=\D^n_{i(k)} t\cdot X(t^n_{i(k+1)})
+\D^n_{i(k+1)} t\cdot X(t^n_{i(k-1)})-(\D^n_{i(k)} t+\D^n_{i(k+1)} t)X(t^n_{i(k)}).
\]

Denote
\[
V_{i(n)}^{(2)}(X,2)=\sum_{k=1}^{i(n)-1} (\D^{(2)n}_{ir,k}X)^2\quad\mbox{and}\quad
\mu_k^n=(\Delta t^n_{i(k)})^{\g+1/2}(\D^n_{i(k)} t)^{\g+1/2}[\D^n_{i(k)} t+\D^n_{i(k+1)} t].
\]
Define
\[
\widehat \g_n=-\frac{1}{2}+\frac{1}{2\ln(p_{j(n)}/m_{i(n)})}\, \ln\frac{V_{{j(n)}}^{(2)}(X,2)}{V_{{i(n)}}^{(2)}(X,2)}\,.
\]

\begin{thm}  Assume that conditions of Proposition $\ref{prop1}$ are satisfied for two sequences of partitions $(\pi_{i(n)})_{n\gs 1}$ and $(\pi_{j(n)})_{n\gs 1}$ of $[0,T]$ with the properties mentioned above. Then
\begin{equation}\label{riba}
V_{\pi_{k(n)}}^{(2)}(X,2)\longrightarrow 2\kap^2\int_0^T g(\ell(t))\,dt\quad\mbox{a.s.}\qquad \mbox{as}\ n\to\infty
\end{equation}
for $k(n)=i(n)$ and for $k(n)=j(n)$.
If sequences of partitions $\{\pi_{i(n)}\}$ and $\{\pi_{j(n)}\}$, $i(n)< j(n)$, are regular or such that $p_{j(n)}/p_{i(n)}\to 0$ as $n\to\infty$,  then
\[
\widehat\g_n\arol{{a.s.}}\g.
\]
\end{thm}
\textbf{Proof}. Proposition $\ref{prop1}$ yields the limit (\ref{riba}). It is evident that
\[
\frac{1}{2m_n^{2\g+1}}\ls\frac{\Delta t^n_{i(k)}}{\mu^n_k}\ls\frac{1}{2p_n^{2\g+1}}
\]
and
\[
\bigg(\frac{p_{i(n)}}{m_{j(n)}}\bigg)^{2\g+1}\frac{V_{{j(n)}}^{(2)}(X,2)}{V_{{i(n)}}^{(2)}(X,2)}\ls \frac{V_{\pi_{j(n)}}^{(2)}(X,2)}{V_{\pi_{i(n)}}^{(2)}(X,2)} \ls \bigg(\frac{m_{i(n)}}{p_{j(n)}}\bigg)^{2\g+1}\frac{V_{{j(n)}}^{(2)}(X,2)}{V_{{i(n)}}^{(2)}(X,2)}\,.
\]
Next, since $\ln(p_{j(n)}/m_{i(n)})\ls0$ and
\[
\frac{m_{i(n)}^{2\g+1} V_{{j(n)}}^{(2)}(X,2)}{p_{j(n)}^{2\g+1}V_{{i(n)}}^{(2)}(X,2)}\Big / \frac{V_{\pi_{j(n)}}^{(2)}(X,2)}{V_{\pi_{i(n)}}^{(2)}(X,2)}\gs 1,
\]
we have
\begin{align*}
\widehat \g_n=&-\frac{1}{2}+\frac{1}{2\ln(p_{j(n)}/m_{i(n)})}\bigg((2\g+1)\ln(p_{j(n)}/m_{i(n)})+\ln\frac{m_{i(n)}^{2\g+1} V_{{j(n)}}^{(2)}(X,2)}{p_{j(n)}^{2\g+1}V_{{i(n)}}^{(2)}(X,2)}\bigg)\\
=&\g+\frac{1}{2\ln(p_{j(n)}/m_{i(n)})}\,\ln\frac{m_{i(n)}^{2\g+1} V_{{j(n)}}^{(2)}(X,2)}{p_{j(n)}^{2\g+1}V_{{i(n)}}^{(2)}(X,2)}\\
=&\g+\frac{1}{2\ln(p_{j(n)}/m_{i(n)})}\,\ln\frac{V_{\pi_{j(n)}}^{(2)}(X,2)}{V_{\pi_{i(n)}}^{(2)}(X,2)}\\
&+\frac{1}{2\ln(p_{j(n)}/m_{i(n)})}\,\ln\bigg(\frac{m_{i(n)}^{2\g+1} V_{{j(n)}}^{(2)}(X,2)}{p_{j(n)}^{2\g+1}V_{{i(n)}}^{(2)}(X,2)}\Big / \frac{V_{\pi_{j(n)}}^{(2)}(X,2)}{V_{\pi_{i(n)}}^{(2)}(X,2)}\bigg)\\
\ls&\g+\frac{1}{2\ln(p_{j(n)}/m_{i(n)})} \,\ln\frac{V_{\pi_{j(n)}}^{(2)}(X,2)}{V_{\pi_{i(n)}}^{(2)}(X,2)}\,.
\end{align*}
In the same way we get
\begin{align}\label{nelyg_n}
\widehat {\g}_n=& -\frac{1}{2}+\frac{1}{2\ln(p_{j(n)}/m_{i(n)})}\bigg((2\g+1)\ln(m_{j(n)}/p_{i(n)})+\ln\frac{p_{i(n)}^{2\g+1} V_{{j(n)}}^{(2)}(X,2)}{m_{j(n)}^{2\g+1}V_{{i(n)}}^{(2)}(X,2)}\bigg)\nonumber\\
=&-\frac{1}{2}+\bigg(\g+\frac{1}{2}\bigg)\frac{\ln(m_{j(n)}/p_{i(n)})}{\ln(p_{j(n)}/m_{i(n)})} +\frac{1}{2\ln(p_{j(n)}/m_{i(n)})}\,\ln\frac{p_{i(n)}^{2\g+1} V_{{j(n)}}^{(2)}(X,2)}{m_{j(n)}^{2\g+1}V_{{i(n)}}^{(2)}(X,2)}\nonumber\\
=&\g+\bigg(\g+\frac{1}{2}\bigg)\frac{\ln(m_{j(n)}/p_{i(n)})-\ln(p_{j(n)}/m_{i(n)})}{\ln(p_{j(n)}/m_{i(n)})} +\frac{1}{2\ln(p_{j(n)}/m_{i(n)})}\,\ln\frac{V_{\pi_{j(n)}}^{(2)}(X,2)}{V_{\pi_{i(n)}}^{(2)}(X,2)}\nonumber\\
&+\frac{1}{2\ln(p_{j(n)}/m_{i(n)})}\,\ln\bigg(\frac{p_{i(n)}^{2\g+1} V_{{j(n)}}^{(2)}(X,2)}{m_{j(n)}^{2\g+1}V_{{i(n)}}^{(2)}(X,2)}\Big / \frac{V_{\pi_{j(n)}}^{(2)}(X,2)}{V_{\pi_{i(n)}}^{(2)}(X,2)}\bigg)\nonumber\\
\gs&\g+\bigg(\g+\frac{1}{2}\bigg)\frac{\ln(m_{j(n)}/p_{i(n)})-\ln(p_{j(n)}/m_{i(n)})}{\ln(p_{j(n)}/m_{i(n)})} +\frac{1}{2\ln(p_{j(n)}/m_{i(n)})}\,\ln\frac{V_{\pi_{j(n)}}^{(2)}(X,2)}{V_{\pi_{i(n)}}^{(2)}(X,2)}\,,
\end{align}
since
\[
\frac{1}{2\ln(p_{j(n)}/m_{i(n)})}\,\ln\bigg(\frac{p_{i(n)}^{2\g+1} V_{{j(n)}}^{(2)}(X,2)}{m_{j(n)}^{2\g+1}V_{{i(n)}}^{(2)}(X,2)}\Big / \frac{V_{\pi_{j(n)}}^{(2)}(X,2)}{V_{\pi_{i(n)}}^{(2)}(X,2)}\bigg)\gs 0
\]
and
\[
\bigg(\g+\frac{1}{2}\bigg)\frac{\ln(m_{j(n)}/p_{i(n)})-\ln(p_{j(n)}/m_{i(n)})}{\ln(p_{j(n)}/m_{i(n)})}\ls 0.
\]
If sequences of partitions $\{\pi_{i(n)}\}$ and $\{\pi_{j(n)}\}$, $i(n)< j(n)$, are regular then the second term in the inequality (\ref{nelyg_n}) is equal to $0$ and
\[
\vert\widehat {\g}_n-\g \vert \ls\frac{1}{2\ln (m_{i(n)}/p_{j(n)})}\,\bigg\vert\ln\frac{V_{\pi_{j(n)}}^{(2)}(X,2)}{V_{\pi_{i(n)}}^{(2)}(X,2)}\bigg\vert\,.
\]
Under conditions of the theorem in the regular case of partitions the statement of the theorem holds. For arbitrary partitions we obtain inequalities
\[
\bigg\vert\widehat {\g}_n-\g -\frac{1}{2\ln (p_{j(n)}/m_{i(n)})}\,\ln\frac{V_{\pi_{j(n)}}^{(2)}(X,2)}{V_{\pi_{i(n)}}^{(2)}(X,2)}\bigg\vert\ls \bigg(\g+\frac{1}{2}\bigg)\frac{\ln(m_{j(n)}/p_{j(n)})+\ln(m_{i(n)}/p_{i(n)})}{\ln(m_{i(n)}/p_{j(n)})}
\]
and
\begin{align*}
\vert\widehat {\g}_n-\g \vert \ls& \bigg\vert\widehat {\g}_n-\g -\frac{1}{2\ln (p_{j(n)}/m_{i(n)})}\,\ln\frac{V_{\pi_{j(n)}}^{(2)}(X,2)}{V_{\pi_{i(n)}}^{(2)}(X,2)}
+\frac{1}{2\ln (p_{j(n)}/m_{i(n)})}\, \ln\frac{V_{\pi_{j(n)}}^{(2)}(X,2)}{V_{\pi_{i(n)}}^{(2)}(X,2)}\bigg\vert\\
\ls&\frac{3}{2}\,\frac{\ln(m_{j(n)}/p_{j(n)})+\ln(m_{i(n)}/p_{i(n)})}{\ln(m_{i(n)}/p_{j(n)})}
+\frac{1}{2\ln (m_{i(n)}/p_{j(n)})}\,\bigg\vert\ln\frac{V_{\pi_{j(n)}}^{(2)}(X,2)}{V_{\pi_{i(n)}}^{(2)}(X,2)}\bigg\vert.
\end{align*}

For arbitrary partitions $\{\pi_{i(n)}\}$ and $\{\pi_{j(n)}\}$, $i(n)< j(n)$, the second term in above inequality goes to $0$ as $\ln (p_{i(n)}/p_{j(n)})\to\infty$, $n\to\infty$. Thus the statement of the theorem holds.

\section{Appendix}

\subsection{Proof of Lemma \ref{orey}}

Assume, without lost of generality, that $0<h<1$. {We first prove that $\widehat\g_*\ls\g_*$, where
\[
\widehat\g_*:=\limsup_{h\downarrow 0}\sup_{\varphi(h)\ls s\ls T-h}\frac{\ln\s_X(s,s+h)}{\ln h}\,, \quad\g_*:=\inf\bigg\{\g>0\dvit \lim_{h\downarrow 0}\sup_{\varphi(h)\ls s\ls T-h}\frac{h^\g}{\s_X(s,s+h)}=0\bigg\}.
\]
} Let $\g>\g_*$. It suffices to show that $\g\gs \widehat\g_*$. By definition of the infimum, there exists a real number $\a$ such that $\g>\a>\g_*$, and
\[
\sup_{\varphi(h)\ls s\ls T-h}\frac{h^\a}{\s_X(s,s+h)} \longrightarrow 0\quad\mbox{as}\ h\downarrow 0.
\]
But
\begin{equation}\label{orey00}
\sup_{\varphi(h)\ls s\ls T-h}\frac{h^\g}{\s_X(s,s+h)}=h^{\g-\a}\sup_{\varphi(h)\ls s\ls T-h}\frac{h^{\a}}{\s_X(s,s+h)} \longrightarrow 0\quad\mbox{as}\ h\downarrow 0
\end{equation}
as the product of two functions tending to $0$. Under the statement
\begin{equation}\label{orey3}
\sup_{\varphi(h)\ls s\ls T-h}\s_X(s,s+h)\longrightarrow 0\qquad\mbox{as}\ h\downarrow 0
\end{equation}
and relation (\ref{orey00}) there exists an $h_0$ such that for all $h\ls h_0<1$
\[
\sup_{\varphi(h)\ls s\ls T-h}\frac{h^\g}{\s_X(s,s+h)}=\frac{h^\g}{\inf_{\varphi(h)\ls s\ls T-h}\s_X(s,s+h)}<1 \quad\mbox{and}\quad \sup_{0\ls s\ls T-h}\s_X(s,s+h)<1.
\]
Moreover,
\[
h^\g<\inf_{\varphi(h)\ls s\ls T-h}\s_X(s,s+h)
\]
for all $h\ls h_0<1$. So
\[
\ln h^\g<\ln \Big(\inf_{\varphi(h)\ls s\ls T-h}\s_X(s,s+h)\Big)\ls\ln \Big(\sup_{\varphi(h)\ls s\ls T-h}\s_X(s,s+h)\Big)
\]
and
\begin{align*}
\g>&\frac{\ln \big(\sup_{\varphi(h)\ls s\ls T-h} \s_X(s,s+h)\big)}{\ln h}=\sup_{\varphi(h)\ls s\ls T-h} \frac{\ln  \s_X(s,s+h)}{\ln h}\\
\gs&\limsup_{h\downarrow 0}\sup_{\varphi(h)\ls s\ls T-h}\frac{\ln\s_X(s,s+h)}{\ln h}= \widehat\g_*\,.
\end{align*}
Thus $\widehat\g_*\ls \g_*$.

Next we prove that $\widehat\g_*\gs\g_*$. Let $\g>\a>\widehat\g_*$. It suffices to show that $\g\gs \g_*$. Under the condition $\a>\widehat\g_*$ and statement (\ref{orey3}) there exists $h_0$ such that for $h\ls h_0<1$
\[
\inf_{\varphi(h)\ls s\ls T-h}\frac{\ln  \s_X(s,s+h)}{\ln h}<\a, \qquad \sup_{0\ls s\ls T-h}\s_X(s,s+h)<1.
\]
This implies the inequality
\[
\ln \Big(\inf_{\varphi(h)\ls s\ls T-h} \s_X(s,s+h)\Big)> \ln h^\a
\]
and
\[
\inf_{\varphi(h)\ls s\ls T-h} \s_X(s,s+h)>h^\a.
\]
Thus
\[
\sup_{\varphi(h)\ls s\ls T-h}\frac{h^\a}{\s_X(s,s+h)}<1.
\]
So
\[
\sup_{\varphi(h)\ls s\ls T-h}\frac{h^\g}{\s_X(s,s+h)}< h^{\g-\a}\longrightarrow 0\quad\mbox{as}\ h\to 0.
\]
Therefore $\g\gs\g_*$.

Now we prove that $\widehat\g^*=\g^*$, where
 \[
\widehat\g^*:=\liminf_{h\downarrow 0}\inf_{\varphi(h)\ls s\ls T-h}\frac{\ln\s_X(s,s+h)}{\ln h},\quad \g^*:=\sup\bigg\{\g>0\dvit \lim_{h\downarrow 0}\inf_{\varphi(h)\ls s\ls T-h}\frac{h^\g}{\s_X(s,s+h)}=+\infty\bigg\}.
\]
We first prove $\widehat\g^*\gs\g^*$. By definition of supremum, there exists a real number $\g$ such that $\g^*>\g$, and
\begin{equation}\label{relation}
\lim_{h\downarrow 0}\inf_{\varphi(h)\ls s\ls T-h}\frac{h^\g}{\s_X(s,s+h)}=+\infty
\end{equation}

It suffices to show that $\widehat\g^*\gs\g$. Under the condition $\g^*>\g$ and statements (\ref{orey3})-(\ref{relation}) there exists $h_0$ such that for $h\ls h_0<1$
\[
\inf_{\varphi(h)\ls s\ls T-h}\frac{h^\g}{\s_X(s,s+h)}>1, \qquad \sup_{0\ls s\ls T-h}\s_X(s,s+h)<1.
\]
Moreover,
\[
h^\g>\sup_{\varphi(h)\ls s\ls T-h}\s_X(s,s+h)\gs\inf_{\varphi(h)\ls s\ls T-h}\s_X(s,s+h)
\]
and
\[
\g\ln h>\ln \inf_{\varphi(h)\ls s\ls T-h}\s_X(s,s+h),\qquad \inf_{\varphi(h)\ls s\ls T-h}\frac{\ln \s_X(s,s+h)}{\ln h}>\g.
\]
So $\widehat\g^*\gs\g$.

{We show that $\g^*\gs\widehat\g^*$.} Assume that $\widehat\g^*>\a>\g$. It sufficient to show that $\g^*>\g$. Under the condition $\widehat\g^*>\a$ and statement (\ref{orey3}) there exists $h_0$ such that for $h\ls h_0<1$
\[
\inf_{\varphi(h)\ls s\ls T-h}\frac{\ln\s_X(s,s+h)}{\ln h}>\a, \qquad \sup_{0\ls s\ls T-h}\s_X(s,s+h)<1.
\]
Moreover,
\[
\sup_{\varphi(h)\ls s\ls T-h}\frac{\ln\s_X(s,s+h)}{\ln h}>\a
\]
and
\[
\ln \Big(\sup_{\varphi(h)\ls s\ls T-h}\s_X(s,s+h)\Big)<\ln h^\a.
\]
Thus
\[
\sup_{\varphi(h)\ls s\ls T-h}\s_X(s,s+h)< h^\a\quad\mbox{and}\quad \inf_{\varphi(h)\ls s\ls T-h}\frac{h^\a}{\s_X(s,s+h)}>1.
\]
Then
\[
\inf_{\varphi(h)\ls s\ls T-h}\frac{h^\g}{\s_X(s,s+h)}>h^{\g-\a}\to \infty
\]
and $\g^*>\g$.

\emph{Acknowledgements}  The author would like to thank the anonymous reviewers for their helpful and constructive comments that greatly contributed to improving the final version of the paper.


\begin{thebibliography}{00}


\bibitem{BGT} T. Bojdecki, L. G. Gorostiza, and A. Talarczyk, Sub-fractional Brownian motion
and its relation to occupation time, \emph{Stat. \& Probab. Lett.} \textbf{69} (2004), 405-419.

\bibitem{begyn1} A. B\'egyn, Quadratic variations along irregular subdivisions for Gaussian
processes, \emph{Electronic Journal of Probability}, \textbf{10} (2005), 691-717.


\bibitem{begyn2} A. B\'egyn, Asymptotic development and central limit theorem for
quadratic variations of Gaussian processes, \emph{Bernoulli} \textbf{13}(3) (2007), 712-753.


\bibitem{chm} P. Cheridito, H. Kawaguchi, and M. Maejima,  Fractional Ornstein-Uhlenbeck processes. \emph{Electronic Journal of Probability}, \textbf{8} (2003), 1-14.



\bibitem{crle} H. Cram\'er and M. R. Leadbetter,  \emph{Stationary and Related Stochastic Processes}, Wiley,
New York, 1967.

\bibitem{dieu} J. Dieudonn\'e, \emph{Foundations of modern analysis}, v. 1, Academic Press, 1969.

\bibitem{gsv} D. Gasbarra, T. Sottinen,  and E. Valkeila, Gaussian bridges. \emph{Stochastic Analysis and Applications: The Abel Symposium 2005.} (Eds. F. Benth, G. Di Nunno, T. Lindstr{\o}m, B. {\O}ksendal and T. Zhang) Abel Symposia, Volume 2, 2007, 361-382, Springer.

\bibitem{glad}  E. G. Gladyshev, A new Limit theorem for stochastic processes with Gaussian increments, \emph{Theory Probab. Appl.}, \textbf{6}(1) (1961),  52-61.



\bibitem{HV} C. Houdr\'e and J. Villa, An example of infinite dimensional quasi-helix. \emph{Contemporary
Mathematics}, \textbf{366} (2003), 195-201.

\bibitem{IL} J. Istas and G. Lang, Quadratic variations and estimation of the local H\"older
index of a Gaussian process. \emph{Ann. Inst. Henri Poincar\'e, Probab. Stat.}, \textbf{33} (1997),   407-436.

\bibitem{kg} R. Klein and E. Gin\'e. On quadratic variation of processes with Gaussian
increments. \emph{Ann. Probab.}, \textbf{3}(4) (1975),  716-721.

\bibitem{ma} R. Malukas, Limit theorems for a quadratic variation of Gaussian processes. \emph{Nonlinear Analysis: Modelling and Control}, \textbf{16}(4) (2011), 435-452.


\bibitem{rn0} R. Norvai\v sa and D.M. Salopek,  Estimating the Orey Index of a
Gaussian Stochastic Process with Stationary Increments: An Application
to Financial Data Set. Canadian Mathematical Society Conference Proceedings,
\textbf{26}  (2000), 353-374.

\bibitem{rn1} R. Norvai\v sa, A coplement to Gladyshev's theorem. \emph{Lithuanian Mathematical Journal}, \textbf{51}(1)  (2011), 26-35.

\bibitem{rn2} R. Norvai\v sa, Gladyshev's theorem for integrals with respect to a Gaussian process, Preprint, 2011, arXiv:1105.1503v1


\bibitem{or} S. Orey,  Gaussian sample functions and the Hausdorff dimension of level crossings. \emph{Z. Wahrsch. verw. Gebiete}, \textbf{15} (1970),  249-256.

\bibitem{rt} F. Russo and C. Tudor, On bifractional Brownian motion. \emph{Stoch. Proc. Appl.}, \textbf{116} (2006), 830-856.

\bibitem{tr} C. Tricot, \emph{Curves and Fractional Dimension}, Springer-Verlag, New York, 1995.






\end{thebibliography}
\end{document}